\documentclass[12pt,a4paper,reqno]{amsart} 
\usepackage{amssymb}
\usepackage{latexsym}
\usepackage{amsmath}
\usepackage{mathrsfs}
\usepackage[usenames,dvipsnames]{xcolor}
\usepackage{nccmath}    
\usepackage[T1]{fontenc} 
\usepackage{cite}
\usepackage{calc}                   

\newcommand{\scal}[2]{\langle #1,#2\rangle}

\newcommand{\rr}[1]{\mathbf R^{#1}}

\newcommand{\cc}[1]{\mathbf C^{#1}}
\newcommand{\zz}[1]{\mathbf Z^{#1}}
\newcommand{\nm}[2]{\Vert #1\Vert _{#2}}
\newcommand{\Nm}[2]{\left \Vert #1 \right \Vert _{#2}}

\newcommand{\sets}[2]{\{ \, #1\, ;\, #2\, \} }
\newcommand{\Sets}[2]{\left \{ \, #1\, ;\, #2\, \right \} }

\newcommand{\ep}{\varepsilon}

\newcommand{\cdo}{\, \cdot \, }

\newcommand{\eabs}[1]{\langle #1\rangle}     

\newcommand{\vrum}{\vspace{0.1cm}}

\newcommand{\rd}{\mathbf{R} ^{d}}
\newcommand{\rdd}{\mathbf{R} ^{2d}}
\newcommand{\zd}{\mathbf{Z}^{d}}
\newcommand{\zdd}{\mathbf{Z}^{2d}}

\newcommand{\nn}[1]{{\mathbf N}^{#1}}

\newcommand{\maclB}{\mathcal B}
\newcommand{\maclE}{\mathcal E}

\newcommand{\maclH}{\mathcal H}

\newcommand{\maclM}{\mathcal M}
\newcommand{\maclS}{\mathcal S}
\newcommand{\mascB}{\mathscr B}

\newcommand{\mascE}{\mathscr E}
\newcommand{\mascF}{\mathscr F}

\newcommand{\mascP}{\mathscr P}
\newcommand{\mascS}{\mathscr S}

\numberwithin{equation}{section}          

\newtheorem{thm}{Theorem}
\numberwithin{thm}{section}

\newcommand{\rubrik}{}
\newtheorem{prop}[thm]{Proposition}
\newtheorem{cor}[thm]{Corollary}
\newtheorem{lemma}[thm]{Lemma}

\theoremstyle{definition}

\newtheorem{defn}[thm]{Definition}

\theoremstyle{remark}
\newtheorem{rem}[thm]{Remark}

\author{Joachim Toft}

\address{Department of Mathematics,
Linn{\ae}us University, V{\"a}xj{\"o}, Sweden}

\email{joachim.toft@lnu.se}

\author{R{\"u}ya {\"U}ster}

\address{Department of Mathematics,
Faculty of Science, \.{I}stanbul University, \.{I}stanbul, Turkey}

\email{ruya.uster@istanbul.edu.tr}

\author{Elmira Nabizadeh}

\address{Department of Mathematics,
Linn{\ae}us University, V{\"a}xj{\"o}, Sweden}

\email{elmira.nabizadehmorsalfard@lnu.se}

\author{Serap {\"O}ztop}

\address{Department of Mathematics,
Faculty of Science, \.{I}stanbul University, \.{I}stanbul, Turkey}
\email{oztops@istanbul.edu.tr}

\title{Continuity and Bargmann mapping properties of
quasi-Banach Orlicz modulation spaces}



\begin{document}

\begin{abstract}
We deduce continuity, compactness and invariance properties for
quasi-Banach Orlicz modulation spaces. We characterize such spaces
in terms of Gabor expansions and by their images under the Bargmann
transform.
\end{abstract}

\maketitle

\par

\section{Introduction}\label{sec0}

\par

In the paper we extend the analysis in \cite{GaSa,Gc2} concerning classical
modulation spaces,  $M^{p,q}_{(\omega)}(\rd)$, and in \cite{SchF} concerning
Banach Orlicz modulation spaces to quasi-Banach Orlicz modulation spaces (quasi-Orlicz
modulation spaces), $M^{\Phi _1,\Phi _2}_{(\omega)}(\rd)$.
Here $\Phi _1$, $\Phi _2$ are quasi-Young functions of certain degrees. We refer
to \cite{Ho1} and Section 1 for notations.

\par

Resembling on classical modulation spaces, Orlicz modulation
spaces are defined by imposing a mixed  $L^{\Phi _1,\Phi _2}_{(\omega)}$
(quasi-)norm condition on the short-time Fourier transforms of the involved
distributions.

\par

In the restricted case when $\Phi _1$ and $\Phi _2$ above are Young functions,
corresponding Orlicz modulation spaces,  $M^{\Phi _1,\Phi _2}_{(\omega)}(\rd)$
were introduced and investigated in \cite{SchF} by Schnackers and F{\"u}hr.
Here it is deduced that for such $\Phi _1$ and $\Phi _2$,
$M^{\Phi _1,\Phi _2}_{(\omega)}(\rd)$ is a Banach space and admit in similar
ways as for classical modulation spaces, characterizations by Gabor expansions.
In \cite{SchF} it is also shown that  $M^{\Phi _1,\Phi _2}_{(\omega)}(\rd)$ is
completely determined by the behaviour of $\Phi _1$ and $\Phi _2$ at origin in the sense that if
\begin{equation}\label{Eq:limcond}
\lim
\limits_{t\rightarrow 0^+} \frac{\Psi _1(t)}{\Phi _1(t)}
\quad\text{and}\quad
\lim
\limits_{t\rightarrow 0^+} \frac{\Psi _2(t)}{\Phi _2(t)}
\end{equation}
exist, then
\begin{equation}\label{Eq:Modinc}
 M^{\Phi _1,\Phi _2}_{(\omega)}(\rd)\subseteq M^{\Psi _1,\Psi _2}_{(\omega)}(\rd)
\end{equation}
with continuous embedding.

\par

In our situation, allowing, more generally, $\Phi _1$ and $\Phi _2$ to be quasi-Young
functions, we show that these and several other continuity properties in
\cite{GaSa,Gc2,PfTo,Toft10,Toft18}, for classical modulation spaces, carry
over to Orlicz modulation spaces.

\par

More precisely, we show that  $M^{\Phi _1,\Phi _2}_{(\omega)}(\rd)$ are quasi-Banach
spaces, and deduce invariance properties
concerning the choices of window functions in the quasi-norms of
the short-time Fourier transforms.
By our general continuity results and similar arguments as for classical modulation
spaces, it follows that the injection map
    \begin{equation*}
           i: M^{\Phi _1,\Phi _2}_{(\omega_1)} (\rd)\rightarrow M^{\Phi _1,\Phi _2}_{(\omega_2)}(\rd)
    \end{equation*}
is compact, if and only if $\omega_2/\omega_1$ tends to zero at infinity. This extends results
in \cite{PfTo} to the Orlicz modulation space case. A part of the analysis concerns investigations
of mapping properties of Orlicz modulation spaces under the Bargmann transform, of
independent interest, given in Section \ref{sec2}. These investigations lead to that the
Bargmann transform is isometric and bijective
from $ M^{\Phi _1,\Phi _2}_{(\omega)} (\rd)$ to certain weighted versions of
$L^{\Phi _1,\Phi _2}_{(\omega)} (\rdd ) \simeq L^{\Phi _1,\Phi _2}_{(\omega)}(\cc d)$
of entire analytic
functions on $\cc d$.

\par

Several of these properties follow from our characterizations of Orlicz modulation
spaces in terms of Gabor expansions, given in Section 4. In fact, here it is proved
that for a distribution $f$, lattice $\Lambda \subseteq \rd$ and
suitable (window) functions  $\phi$ and $\psi$ on $\rd$, then the analysis
and synthesis operators,
$$
(C_\phi f)= \{ (V_\phi f)(j,\iota )\} _{j,\iota \in \Lambda}
\quad \text{and}\quad
(D_\psi c) = \sum _{j,\iota \in \Lambda}  c(j,\iota)e^{i\scal \cdo \iota}\psi (\cdo -j)
$$
are continuous between the spaces $M^{\Phi _1, \Phi _2}_{(\omega )}(\rr d)$ and
$\ell ^{\Phi _1, \Phi _2}_{(\omega )}(\Lambda \times \Lambda )$. These properties
leads to that Orlicz modulation spaces possess Gabor properties in the sense
that for each suitable (ultra-)distribution $f$ we have
$$
f(x)=\sum \limits _{j,\iota \in \Lambda}(V_\phi f)(j,\iota)e^{i \scal x\iota}\psi(x-j),
$$
and that
\begin{equation}\label{Eq:GabEqui}
f\in M^{\Phi _1,\Phi _2}_{(\omega)} (\rd)\quad \Leftrightarrow \quad
\{ (V_\phi f)(j,\iota)\}_{j,\iota \in \Lambda} \in \ell^{\Phi _1, \Phi _2}
_{(\omega)}(\Lambda\times \Lambda),
\end{equation}
provided that the lattice $\Lambda$ is enough dense.
In particular, the Gabor analysis for classical modulation spaces in
\cite{GaSa,Gc2} are extended to quasi-Orlicz modulation spaces.

\par

The paper is organized as follows. In Section \ref{sec1} we recall some
basic properties on Gelfand-Shilov spaces, weight functions, Pilipovi{\'c} spaces,
Orlicz spaces, and introduce quasi-Orlicz modulation spaces. Here
we also recall some properties for classical modulation spaces concerning
Gabor analysis, images under pseudo-differential operators and under the Bargmann
transform.

\par

In Section \ref{sec2} we deduce mapping properties of Orlicz modulation spaces
under the Bargmann transform. At the same time we prove that they are
complete, and thereby quasi-Banach spaces. In Section \ref{sec3} we obtain convolution
estimates for quasi-Orlicz spaces.

\par

In Section \ref{sec4}, we apply the convolution results in Section \ref{sec3} to extend
the Gabor analysis for classical modulation spaces to quasi-Orlicz modulation
spaces. In particular we deduce \eqref{Eq:GabEqui} for quasi-Young functions
$\Phi _1$ and $\Phi _2$. This extends \cite[Theorem 9]{SchF} by Schnacker and F{\"u}hr.
We also apply the analysis to deduce basic continuity properties of
such spaces. For example we show invariance properties with respect to the choice of
window function, and use the equivalence \eqref{Eq:GabEqui} to show that
\eqref{Eq:limcond}
leads to \eqref{Eq:Modinc}, also when $\Phi _1$ and $\Phi _2$ are
quasi-Young functions.

\par

\section{Preliminaries}\label{sec1}

\par

In this section we make a review of some basic facts.
In the first part we recall the definition and explain some well-known
facts about Gelfand-Shilov and Pilipovi{\'c} spaces and their spaces of
(ultra-)distributions. Thereafter we consider (mixed) Orlicz and quasi-Orlicz
spaces and explain some basic properties. Our family of quasi-Orlicz spaces
contain the family of Orlicz spaces, but is not so general compared to
corresponding families in e.{\,}g. \cite{HaH} by Harjulehto
and H{\"a}st{\"o}.

\par

Then we introduce and discuss basic properties of quasi-Banach Orlicz
modulation spaces, which are obtained by imposing quasi-Orlicz norm
estimates on the short-time Fourier transforms of the involved functions
and distributions. Finally we recall some basic facts in Gabor frame
theory, and for the Bargmann transform.

\par

\subsection{Gelfand-Shilov spaces}\label{subsec1.1}
 We start by discussing Gelfand-Shilov spaces and their properties.
 Let $0<s\in \mathbf R$ be fixed. Then the Gelfand-Shilov
space $\mathcal S_{s}(\rr d)$
 ($\Sigma _{s}(\rr d)$) of Roumieu type (Beurling type) with parameter $s$
 consists of all $f\in C^\infty (\rr d)$ such that
 \begin{equation}\label{gfseminorm}
 \nm f{\mathcal S_{s,h}}\equiv \sup \frac {|x^\beta \partial ^\alpha
 f(x)|}{h^{|\alpha  + \beta |}(\alpha ! \beta !)^s}
 \end{equation}
 is finite for some $h>0$ (for every $h>0$). Here the supremum should be taken
 over all $\alpha ,\beta \in \mathbf N^d$ and $x\in \rr d$. We equip
 $\mathcal S_{s}(\rr d)$ ($\Sigma _{s}(\rr d)$) by the canonical inductive limit
 topology (projective limit topology) with respect to $h>0$, induced by
 the semi-norms in \eqref{gfseminorm}.

 \par

 For any $s,s_0\ge \frac 12$ such that $s_0<s$ we have
\begin{equation}\label{GSembeddings}
 \begin{alignedat}{3}
 \maclS _{s_0}(\rr d)
 &\hookrightarrow &
 \Sigma _{s}(\rr d)
 &\hookrightarrow &
 \maclS _s(\rr d)
 &\hookrightarrow
 \mascS (\rr d),
  \\[1ex]
 \mascS '(\rr d)
 &\hookrightarrow &
 \maclS _s' (\rr d)
 &\hookrightarrow &
 \Sigma _{s}'(\rr d)
 &\hookrightarrow
 \maclS _{s_0}'(\rr d),
 \end{alignedat}
 \end{equation}
 with dense embeddings.
 Here $A\hookrightarrow B$ means that
 the topological spaces $A$ and $B$ satisfy $A\subseteq B$ with
 continuous embeddings.
 The space $\Sigma _s(\rr d)$ is a Fr{\'e}chet space
 with seminorms $\nm \cdot {\mathcal S_{s,h}}$, $h>0$. Moreover,
 $\Sigma _s(\rr d)\neq \{ 0\}$, if and only if $s>1/2$, and
 $\maclS _s(\rr d)\neq \{ 0\}$, if and only
 if $s\ge 1/2$.

 \medspace

 The \emph{Gelfand-Shilov distribution spaces} $\mathcal S_{s}'(\rr d)$
and $\Sigma _s'(\rr d)$ are the dual spaces of $\mathcal S_{s}(\rr d)$ and
$\Sigma _s(\rr d)$, respectively.  As for the Gelfand-Shilov spaces there
 is a canonical projective limit topology (inductive limit topology) for $\maclS _{s}'(\rr d)$
 ($\Sigma _s'(\rr d)$). (Cf. \cite{GS, Pil1, Pil3}.)

 \par

 From now on we let $\mathscr F$ be the Fourier transform which
 takes the form
 $$
 (\mathscr Ff)(\xi )= \widehat f(\xi ) \equiv (2\pi )^{-\frac d2}\int _{\rr
 {d}} f(x)e^{-i\scal  x\xi }\, dx
 $$
 when $f\in L^1(\rr d)$. Here $\scal \cdo \cdo$ denotes the usual
 scalar product on $\rr d$. The map $\mathscr F$ extends
 uniquely to homeomorphisms on $\mathscr S'(\rr d)$,
$\mathcal S_s'(\rr d)$ and on $\Sigma _s'(\rr d)$. Furthermore,
 $\mascF$ restricts to
 homeomorphisms on $\mathscr S(\rr d)$,
 $\mathcal S_s(\rr d)$ and on $\Sigma _s(\rr d)$,
 and to a unitary operator on $L^2(\rr d)$.

 \par

Gelfand-Shilov spaces can in convenient
 ways be characterized in terms of estimates of the functions and their Fourier
 transforms. More precisely, in \cite{ChuChuKim, Eij} it is proved that
 if $f\in \mascS '(\rr d)$ and $s>0$, then $f\in \maclS _s(\rr d)$
 ($f\in \Sigma _s(\rr d)$), if and only if
 \begin{equation}\label{Eq:GSFtransfChar}
 |f(x)|\lesssim e^{-r|x|^{\frac 1s}}
 \quad \text{and}\quad
 |\widehat f(\xi )|\lesssim e^{-r|\xi |^{\frac 1s}},
 \end{equation}
 for some $r>0$ (for every $r>0$).
 Here $g_1 \lesssim g_2$ means that $g_1(\theta ) \le c \cdot  g_2(\theta )$
 holds uniformly for all $\theta$
 in the intersection of the domains of $g_1$ and $g_2$
 for some constant $c>0$, and we
 write $g_1\asymp g_2$
 when $g_1\lesssim g_2 \lesssim g_1$.

 \par

 Gelfand-Shilov spaces and their distribution spaces can also
 be characterized by estimates of short-time Fourier
 transforms, (see e.{\,}g. \cite{GZ,Toft18}).
 More precisely, let $\phi \in \maclS _s (\rr d)$ be fixed. Then the \emph{short-time
 Fourier transform} $V_\phi f$ of $f\in \maclS _s '
 (\rr d)$ with respect to the \emph{window function} $\phi$ is
 the Gelfand-Shilov distribution on $\rr {2d}$, defined by
 \begin{equation} \label{Eq:ShorttimeF}
      V_\phi f(x,\xi )  =
 \mascF (f \, \overline {\phi (\cdo -x)})(\xi ).
 \end{equation}
 If $f ,\phi \in \maclS _s (\rr d)$, then it follows that
 $$
 V_\phi f(x,\xi ) = (2\pi )^{-\frac d2}\int_{\rr d} f(y)\overline {\phi
 (y-x)}e^{-i\scal y\xi}\, dy .
 $$

\par

\subsection{Weight functions}\label{subsec1.2}
A \emph{weight} or \emph{weight function} on $\rr d$ is a
positive function $\omega
\in  L^\infty _{loc}(\rr d)$ such that $1/\omega \in  L^\infty _{loc}(\rr d)$.
The weight $\omega$ is called \emph{moderate},
if there is a positive weight $v$ on $\rr d$ such that
\begin{equation}\label{moderate}
\omega (x+y) \lesssim \omega (x)v(y),\qquad x,y\in \rr d.
\end{equation}
If $\omega$ and $v$ are weights on $\rr d$ such that
\eqref{moderate} holds, then $\omega$ is also called
\emph{$v$-moderate}.
We note that \eqref{moderate}
implies that $\omega$ fulfills
the estimates
\begin{equation}\label{moderateconseq}
v(-x)^{-1}\lesssim \omega (x)\lesssim v(x),\quad x\in \rr d.
\end{equation}
We let $\mascP _E(\rr d)$ be the set of all moderate weights on $\rr d$.

\par

It can be proved that if $\omega \in \mascP _E(\rr d)$, then
$\omega$ is $v$-moderate for some $v(x) = e^{r|x|}$, provided the
positive constant $r$ is large enough (cf. \cite{Gc2.5}). That is,
\eqref{moderate} implies
\begin{equation}\label{Eq:weight0}
\omega (x+y) \lesssim \omega(x) e^{r|y|},\quad x,y\in \rr d
\end{equation}
for some $r>0$. In particular, \eqref{moderateconseq} shows that
for any $\omega \in \mascP_E(\rr d)$, there is a constant $r>0$ such that
$$
e^{-r|x|}\lesssim \omega (x)\lesssim e^{r|x|},\quad x\in \rr d.
$$

\par

We say that $v$ is
\emph{submultiplicative} if $v$ is even and
\eqref{moderate}
holds with $\omega =v$. In the sequel, $v$ and $v_j$ for
$j\ge 0$, always stand for submultiplicative weights if
nothing else is stated.

\par

We let $\mascP ^{0} _E(\rd)$ be the set of all $\omega\in \mascP _E(\rd)$
such that \eqref{Eq:weight0} holds for every $r>0$. We also let $\mascP (\rd)$
be the set of all $\omega\in \mascP _E(\rd)$ such that
$$
\omega (x+y) \lesssim \omega(x) (1+|y|)^r
$$
for some $r>0$.
Evidently,
$$
\mascP (\rd) \subseteq \mascP ^{0} _E(\rd) \subseteq \mascP _E(\rd).
$$

\par

\subsection{Pilipovi{\'c} Spaces}

\par

Some of our investigating later on are performed in the framework of the
Pilipovi{\'c} space $\maclH_\flat (\rd)$ and its dual $\maclH_\flat' (\rd)$.

\par

We recall from \cite{Toft18} that the Pilipovi{\'c} space $\maclH_\flat (\rd)
= \maclH_{\flat_1}(\rd)$ is the set of all Hermite series expansions
\begin{equation} \label{Eq:Hermite1}
f=\sum \limits _{\alpha \in \mathbf{N}^d} c_f(\alpha) h_\alpha
\end{equation}
such that
\begin{equation}\label{Eq:Hermite2}
    | c_f(\alpha)|\lesssim r^{|\alpha|} \alpha!^{-\frac 12}
\end{equation}
for some $r>0$. Here
$h_\alpha$ is the Hermite function of order $\alpha>0$ which is given by
$$
h_\alpha (x)= \pi^{-\frac d4} (-1)^{|\alpha|} (2^{|\alpha|} \alpha !)^{-\frac 12}
e^{\frac 12\cdot {|x|^2}} (\partial^\alpha e^{-|x|^2}),\quad \alpha \in \nn d.
$$
In the same way, $\maclH_\flat' (\rd)$ consists of all formal Hermite
series expansion \eqref{Eq:Hermite1} such that
\begin{equation}\label{Eq:dualHermite}
| c_f(\alpha)|\lesssim r^{|\alpha|}\alpha!^{\frac 12}
\end{equation}
for every $r>0$.
The topologies of $\maclH_\flat (\rd)$ and $\maclH_\flat' (\rd)$ are given by suitable
inductive limit respectively projective limit topologies with respect to $r$ in
\eqref{Eq:Hermite2} and \eqref{Eq:dualHermite}. (See \cite{Toft18} for details.)
By identifying elements in $\mathscr{S}(\rd),\ \maclS_s(\rd)$ and $\Sigma_s(\rd)$ we get
the dense embeddings
\begin{multline}
    \maclH_\flat (\rd) \hookrightarrow \maclS_{1/2}(\rd)\hookrightarrow
\Sigma_s(\rd) \hookrightarrow
 \maclS_s(\rd) \hookrightarrow
 \mathscr{S}(\rd)
 \\[1ex]
\hookrightarrow  \mathscr{S}'(\rd)
\hookrightarrow \maclS_s'(\rd) \hookrightarrow
 \Sigma_s'(\rd) \hookrightarrow \maclS_{1/2}'(\rd) \hookrightarrow \maclH_\flat' (\rd),
\quad
s> \frac{1}{2} .
\end{multline}

\par

We also have
\begin{equation}\label{Eq:innerpro}
    (f,g)_{L^2(\rd)} = \sum\limits_{\alpha \in \mathbf{N}^ d } c_f(\alpha) \overline{c_g(\alpha)},
\end{equation}
when $f,g\in L^2(\rd)$.
By letting the $L^2$-form $(f,g)_{L^2(\rd)}$ be equal to the right-hand
side of \eqref{Eq:innerpro} when $f\in \maclH_\flat (\rd)$ and $g\in \maclH'_\flat (\rd)$,
it follows that $\maclH'_\flat (\rd)$ is the dual of $\maclH_\flat (\rd)$ through a unique
extension of the $L^2$-form on $\maclH_\flat (\rd) \times \maclH_\flat (\rd)$ to
$\maclH_\flat (\rd) \times \maclH'_\flat (\rd)$ or $\maclH'_\flat (\rd) \times
\maclH_\flat (\rd)$.

\par

For future references we remark that if $\phi(x)= \pi^{-\frac d4}e^{-\frac 12\cdot {|x|^2}}$
and $f \in \maclH'_\flat (\rd)$, then the short-time Fourier transform \eqref{Eq:ShorttimeF}
makes sense as a smooth functions in view of (2.25) and Theorem 4.1 in \cite{Toft18}.

\par

\subsection{Quasi-Banach Spaces}

\par

We recall that a quasi-norm $\nm {\cdo}{\mascB}$ of order $r_0\in (0,1]$ on the
vector-space $\mascB$ over $\mathbf C$ is a nonnegative functional on
$\mascB$ which satisfies
\begin{alignat}{2}
 \nm {f+g}{\mascB} &\le 2^{\frac 1{r_0}-1}(\nm {f}{\mascB} + \nm {g}{\mascB}), &
\quad f,g &\in \mascB ,
\label{Eq:WeakTriangle1}
\\[1ex]
\nm {\alpha \cdot f}{\mascB} &= |\alpha| \cdot \nm f{\mascB},
& \quad \alpha &\in \mathbf{C},
\quad  f \in \mascB
\notag
\intertext{and}
   \nm f {\mascB} &= 0\quad  \Leftrightarrow \quad f=0. & &
\notag
\end{alignat}
The space $\mascB$ is then called a quasi-norm space. A complete
quasi-norm space is called a quasi-Banach space. If $\mascB$
is a quasi-Banach space with
quasi-norm satisfying \eqref{Eq:WeakTriangle1}
then by \cite{Aik,Rol} there is an equivalent quasi-norm to $\nm \cdo {\mascB}$
which additionally satisfies
\begin{align}\label{Eq:WeakTriangle2}
\nm {f+g}{\mascB}^{r_0} \le \nm {f}{\mascB}^{r_0} + \nm {g}{\mascB}^{r_0},
\quad f,g \in \mascB .
\end{align}
From now on we always assume that the quasi-norm of the quasi-Banach space $\mascB$
is chosen in such way that both \eqref{Eq:WeakTriangle1} and \eqref{Eq:WeakTriangle2}
hold. The space $\mascB$ is then also called an $r_0$-Banach space.

\par

\subsection{Orlicz Spaces}

\par

Next we define and recall some basic facts for (quasi-) Orlicz spaces.
(See \cite{Rao,HaH}.)
First we give the definition of Young functions and quasi-Young functions.

\par

\begin{defn}\label{convex f.}
A function $\Phi:\mathbf R \rightarrow
\mathbf R \cup \{ \infty\}$ is called \emph{convex} if
\begin{equation*}
\Phi(s_1 t_1+ s_2 t_2)
\leq s_1 \Phi(t_1)+s_2\Phi(t_2)
\end{equation*}
when
$s_j,t_j\in \mathbf{R}$
satisfy $s_j \ge 0$ and
$s_1 + s_2 = 1,\ j=1,2$.
\end{defn}
We observe that $\Phi$ might not
 be continuous, because we permit
 $\infty$ as function value. For example,
$$\Phi(t)=
\begin{cases}
  c,&\text{when}\ t \leq a
  \\[1ex]
   \infty,&\text{when}\ t>a
\end{cases}$$
is convex but discontinuous at $t=a$.

\par

\begin{defn}\label{Young func.}
Let $r_0\in(0,1]$, $\Phi _0$ and
$\Phi$ be functions from $[0,\infty)$ to $[0,\infty]$.
Then $\Phi _0$ is called a \emph{Young function} if
\begin{enumerate}
  \item $\Phi _0$ is convex,

  \vrum

  \item $\Phi _0(0)=0$,

  \vrum

  \item $\lim
\limits_{t\rightarrow\infty} \Phi _0(t)=+\infty$.
\end{enumerate}
The function $\Phi$ is called
\emph{$r_0$-Young function} or
\emph{quasi-Young function of
order $r_0$}, if $\Phi(t)=\Phi _0 (t^{r_0})$, $t \geq 0$,
for some Young function $\Phi _0$.
\end{defn}

\par

It is clear that $\Phi$ in
Definition \ref{Young func.} is
non-decreasing,
because if $0\leq t_1\leq t_2$
and $s\in [0,1]$ is chosen such
that $t_1=st_2$, then
\begin{equation*}
    \Phi(t_1)=\Phi(st_2+(1-s)0)
    \leq s\Phi(t_2)+(1-s)\Phi(0)
    \leq \Phi(t_2),
\end{equation*}
since $\Phi(0)=0$, $\Phi(t_2)\geq 0$ and $s\in [0,1]$.

\par

\begin{defn}
Let $\Omega \subseteq \rd$, $(\Omega,\Sigma,\mu)$ be a Borel measure
space, $\Phi _0$ be a
Young function and let $\omega_0 \in \mascP _E(\rr d)$.
\begin{enumerate}
    \item $L^{\Phi _0}_{(\omega_0)}(\mu)$ consists
    of all $\mu$-measurable functions
    $f:\Omega \rightarrow
    \mathbf C$ such that
    $$
    \Vert f\Vert_{L^{\Phi _0}_{(\omega_0)}}=\inf  \Sets{\lambda>0}{\int_\Omega \Phi _0
    \left (
    \frac{|f(x) \cdot \omega_0 (x)|}{\lambda}
    \right )
    d\mu (x)\leq 1}
    $$
is finite.
%
%

\vrum

\item Let $\Phi$ be a quasi-Young
function of order $r_0\in (0,1]$,
given by $\Phi(t)=\Phi _0(t^{r_0})$,
$t\geq 0$, for some Young function $\Phi _0$. Then $L^ \Phi _{(\omega_0)}(\mu)$ consists
of all $\mu$-measurable functions
$f:\Omega \rightarrow \mathbf C$ such that
$$
\Vert f\Vert_{L^{\Phi}_{(\omega_0)}}=(\Vert|f
\cdot \omega_0 |^{r_0}\Vert_{L^{\Phi _0}})^{1/r_0}
$$
is finite.
\end{enumerate}
\end{defn}

\par

\begin{rem}
Let $\Phi$, $\Phi _0$ and $\omega_0$ be the
same as in Definition \ref{Young func.}.
Then it follows by straight-forward computation that
$$
\Vert f\Vert_{L^\Phi _{(\omega_0)}}=\inf\Sets{\lambda>0}
{\int_\Omega \Phi _0
    \left (
\frac{|f(x) \cdot \omega_0 (x)|^{r_0}}{\lambda^{r_0}}
\right)
d\mu(x)\leq 1}.
$$
\end{rem}

\par

\begin{defn}\label{d1}
Let $(\Omega_j ,\Sigma_j ,\mu_j )$ be Borel
measure spaces, with $\Omega_j \subseteq \rr {d_j}$, $r_0\in (0,1]$, $\Phi _{j}$ be
$r_0$-Young functions, $j=1,2$ and let $\omega \in \mascP _E (\rr {d_1+d_2} )$. Then
the mixed quasi-norm Orlicz space ${L^{\Phi _1, \Phi _2}_{(\omega)}} = {L^{\Phi_1, \Phi _2}_{(\omega)}}(\mu_1 \otimes \mu_2)$ consists of all $\mu_1 \otimes
\mu_2$-measurable functions $f:\Omega_1 \times \Omega_2 \rightarrow
\mathbf C$ such that
$$
\Vert f\Vert_{L^{\Phi _1, \Phi _2}_{(\omega)}} \equiv
\Vert f_{1,\omega}\Vert_{L^{\Phi _2}},
$$
is finite, where
$$
f_{1,\omega}(x_2)=\Vert f(\cdo ,x_2) \omega(\cdo, x_2)\Vert_{L^{\Phi _{1}}}.
$$
\end{defn}

\par

If $r_0=1$ in Definition \ref{d1}, then $L^{\Phi _1, \Phi _2}_{(\omega)}
(\mu_1 \otimes \mu_2)$ is a Banach space and is called a mixed norm Orlicz space.

\par

\begin{rem}
Suppose $\Phi _j$ are quasi-Young
functions of order $q_j\in (0,1]$, $j=1,2$.
Then both $\Phi _{1}$ and $\Phi _{2}$
are quasi-Young functions of order
$r_0=\min(q_1,q_2)$.
\end{rem}

\par

\begin{rem}
Let $\omega$, $\mu _1$ and $\mu _2$ be as in Definition \ref{d1}.
For $p\in (0,\infty ]$, let $\Phi _p(t) = \frac {t^p}{p}$ when $p<\infty$, and set
$\Phi _\infty (t)=0$ when $0\le t\le 1$ and $\Phi _\infty (t)=\infty$
when $t>1$. Then it is well-known that
$L^{\Phi _p ,\Phi _q}_{(\omega )}=L^{p,q}_{(\omega )}$ with equality in
quasi-norms. Hence the family of quasi-Orlicz spaces contain the usual Lebesgue
spaces and mixed quasi-normed spaces of Lebesgue types.
\end{rem}

\par

\begin{rem}
Let $r_0$, $\Phi _j$, $\mu _j$ and $\omega$ be as in Definition \ref{d1}.
Then
\begin{equation}\label{Eq:OrliczNormTransfers}
\nm f{L^{\Phi _1, \Phi _2}_{(\omega )}}
=
\left (
\nm {|f|^{r_0}}{L^{\Phi _{0,1},\Phi _{0,2}}_{(\omega )}}
\right ) ^{\frac 1{r_0}}
\end{equation}
for some Young functions $\Phi _{0,1}$ and $\Phi _{0,2}$. It follows that
$L^{\Phi _1, \Phi _2}_{(\omega )}(\mu _1\otimes \mu _2)$ is an $r_0$-Banach space.

\par

In fact, the completeness of $L^{\Phi _1, \Phi _2}_{(\omega )}(\mu _1\otimes \mu _2)$
follows from
\eqref{Eq:OrliczNormTransfers} and the completeness of
$L^{\Phi _{0,1},\Phi _{0,2}}_{(\omega )}(\mu _1\otimes \mu _2)$. Furthermore by
\eqref{Eq:OrliczNormTransfers} and the fact that $r_0\in (0,1]$
we get for every $f,g\in
L^{\Phi _1, \Phi _2}_{(\omega )}(\mu _1\otimes \mu _2)$ that
\begin{multline*}
\nm {f+g}{L^{\Phi _1, \Phi _2}_{(\omega )}}^{r_0}
=
\nm {|f+g|^{r_0}}{L^{\Phi _{0,1},\Phi _{0,2}}_{(\omega )}}
\le
\nm {|f|^{r_0}+|g|^{r_0}}{L^{\Phi _{0,1},\Phi _{0,2}}_{(\omega )}}
\\[1ex]
=
\nm {|f|^{r_0}}{L^{\Phi _{0,1},\Phi _{0,2}}_{(\omega )}}
+
\nm {|g|^{r_0}}{L^{\Phi _{0,1},\Phi _{0,2}}_{(\omega )}}
=
\nm {f}{L^{\Phi _{1},\Phi _{2}}_{(\omega )}}^{r_0}
+
\nm {g}{L^{\Phi _{1},\Phi _{2}}_{(\omega )}}^{r_0},
\end{multline*}
which shows that $\nm \cdo{L^{\Phi _{1},\Phi _{2}}_{(\omega )}}$ is
a quasi-norm of order $r_0$, giving the assertion.
\end{rem}

\par

In what follows let $\ell_0'(\Lambda )$ be the set of all
formal sequences
$$
\{ a(n)\} _{n\in \Lambda} =
\sets{a(n)}{n\in \Lambda}\subseteq \mathbf C,
$$
and let $\ell_0 (\Lambda )$ be the set of all
sequences $\{ a(n)\} _{n\in \Lambda}$ such that $a(n)\neq 0$ for at
most finite numbers of $n$.

\par

\begin{rem}
Let $\Lambda \subseteq \rr d$ be a lattice, $\Phi, \Phi _1$ and $\Phi _2$
be $r_0$-Young functions, $\omega_0, v_0
\in \mascP_E (\rd)$ and  $\omega, v \in \mascP_E (\rdd )$ be such that $\omega
_0$ and $\omega$ are $v_0$- respectively $v$-moderate. (In the sequel it is
understood that all lattices contain $0$.) Then we set
$$
L^{\Phi}_{(\omega_0)}(\rd) = L^{\Phi}_{(\omega_0)}(\mu)
\quad
\text{and}
\quad
L^{\Phi _1, \Phi _2}_{(\omega_0)}(\rdd ) =
L^{\Phi _1,\Phi _2}_{(\omega_0)}(\mu \otimes \mu),
$$
when $\mu$ is the Lebesgue measure on $\rr d$.
If instead
$\mu$ is the standard (Haar) measure on $\Lambda$, i.e.
$\mu(n)=1,\ n\in \Lambda$, and
$$
\ell^{\Phi}_{(\omega)}(\Lambda ) = \ell^{\Phi}_{(\omega)}(\mu)
\quad
\text{and}
\quad
\ell^{\Phi _1, \Phi _2}_{(\omega)}(\Lambda \times \Lambda )
= \ell^{\Phi _1,\Phi _2}_{(\omega)}(\mu \otimes \mu).
$$
Evidently, $\ell^{\Phi _1, \Phi _2}
_{(\omega)}(\Lambda \times \Lambda )\subseteq
\ell _0'(\Lambda \times \Lambda )$.
\end{rem}

\par

\begin{lemma}\label{Lemma:T}
Let $\Phi, \Phi _j$ be Young functions, $j=1,2$, $\omega_0, v_0 \in \mascP_E (\rd)$
and $\omega, v \in \mascP_E (\rdd )$ be such that $\omega_0$ is $v_0$-moderate
and $\omega$ is $v$-moderate. Then $L^{\Phi}_{(\omega_0)}(\rd)$ and
$L^{\Phi _{1} \Phi _{2}}_{(\omega)}(\rr{2d})$ are
 invariant under translations, and
$$
\Vert f(\cdo - x)\Vert_{L^\Phi _{(\omega_0)}} \lesssim
\Vert f\Vert_{L^\Phi _{(\omega_0)}} v_0(x), \quad f\in L^\Phi _{(\omega_0)}(\rd),\ x\in \rd\;,
$$
and
$$
\Vert f(\cdo - (x,\xi))\Vert_{L^{\Phi _{1}, \Phi _{2}}_{(\omega)}}
\lesssim
\Vert f\Vert_{L^{\Phi _{1}, \Phi _{2}}_{(\omega)}}v(x,\xi),
\quad f\in L^{\Phi _{1}, \Phi _{2}}_{(\omega)} (\rdd ),\ (x,\xi ) \in \rdd .
$$
\end{lemma}

\par

We give a proof of the second statement in Lemma \ref{Lemma:T}.

\par

\begin{proof}
We have $\Phi _j(t) = \Phi _{0,j}(t^{r_0}),\ t\geq 0$, for some $r_0\in (0,1]$ and
Young functions $\Phi _{0,j},\ j=1,2$. This gives
\begin{multline*}
\nm {f(\cdo - (x,\xi)}{L^{\Phi _{1}, \Phi _{2}}_{(\omega)}}=
\left(
\nm{|f(\cdo - (x,\xi))\omega|^{r_0}}{L^{\Phi _{0,1}, \Phi _{0,2}}}
\right)^{\frac{1}{r_0}}
\\[1ex]
\lesssim
\left(
\nm{|f(\cdo - (x,\xi))\omega(\cdo - (x,\xi))v(x,\xi)|^{r_0}}{L^{\Phi _{0,1}, \Phi _{0,2}}}
\right)^{\frac{1}{r_0}}
\\[1ex]
=
\left(
\nm{|f \cdot \omega|^{r_0}}{L^{\Phi _{0,1}, \Phi _{0,2}}}
\right)^{\frac{1}{r_0}}
\cdot v(x,\xi)=
\nm{f}{L^{\Phi _{1}, \Phi _{2}}_{(\omega)}}\cdot
v(x,\xi).
\end{multline*}
Here the inequality follows from the fact that $\omega$ is $v$- moderate, and the last two
relations follow from the definitions.
\end{proof}

We refer to \cite{SchF,Rao,HaH} for more facts about Orlicz spaces.

\par

\subsection{Orlicz modulation spaces}\label{subsec1.3}

\par

Before considering Orlicz modulation spaces, we recall the definition of classical
modulation spaces. (Cf. \cite{F1,Fei5}.)

\par

\begin{defn}\label{Def:Orliczmod}
Let $\phi(x) = \pi ^{-\frac{d}{4}}e^{-\frac 12\cdot {|x|^2}},\ x\in \rd$, $p,q\in (0,\infty]$ and
$\omega $ be a weight on $\rdd $.
Then the \emph{modulation spaces} $M^{p,q}_{(\omega)}(\rd)$
is set of all $f\in \maclS _{1/2}'
(\rr d)$ such that $V_\phi f\in L^{p,q}_{(\omega)}(\rr {2d})$.
We equip these spaces with the quasi-norm
\begin{equation*}
\nm f{M^{p,q}_{(\omega)}} \equiv \nm {V_\phi f}{L^{p,q}_{(\omega)}}.
\end{equation*}
Also let $\Phi, \Phi _1,\Phi _2$ be quasi-Young functions. Then the \emph{Orlicz modulation spaces}
$M^{\Phi}_{(\omega)} (\rd)$ and  $M^{\Phi _{1}, \Phi _{2}}_{(\omega)}(\rd)$ are given by
\begin{equation}\label{Eq:Orliczmod1}
M^{\Phi}_{(\omega)}(\rd)=
\sets{f \in \maclH'_{\flat}(\rd)} {V_\phi f\in
L^{\Phi}_{(\omega)} (\rdd )}
\end{equation}
and
\begin{equation}\label{Eq:Orliczmod2}
M^{\Phi _{1}, \Phi _{2}}_{(\omega)}(\rd)=
\sets{f\in \maclH'_{\flat}(\rd)} { V_\phi f\in
L^{\Phi _{1} ,\Phi _{2}} _{(\omega)}(\rdd )}.
\end{equation}
The quasi-norms on $M^{\Phi}_{(\omega)}(\rd)$ and
$M^{\Phi _{1}, \Phi _{2}}_{(\omega)}(\rd)$ are given by
\begin{equation}\label{NE}
\Vert f\Vert_{M^{\Phi}_{(\omega)}} =\Vert V_\phi f\Vert_{L^{\Phi}_{(\omega)}}
\end{equation}
and
\begin{equation}\label{MNE}
\nm f {M^{\Phi _{1}, \Phi _{2}}_{(\omega)}}
=\nm {V_\phi f} {L^{\Phi _{1}, \Phi _{2}}_{(\omega)}}.
\end{equation}
\end{defn}

\par

For conveniency we set $M^{p,q}(\rd)=M^{p,q}_{(\omega)}(\rd)$ when
$\omega (x,\xi)=1$, and we set
$M^p=M^{p,p}$ and $M^p_{(\omega)}=M^{p,p}_{(\omega)}$.

 \par

We notice that \eqref{NE} and \eqref{MNE} are norms when $\Phi, \Phi _1$ and $\Phi _2$
are Young functions.
If $\omega \in \mascP_E(\rdd )$ as in Definition \ref{Def:Orliczmod},
then we prove later on that the conditions
$$
\nm {V_\phi f} {L^{\Phi _{1} \Phi _{2}}_{(\omega)}} <\infty
\quad \text{and}\quad
\nm {V_\phi f} {L^{\Phi}_{(\omega)}}<\infty
$$
are independent of the choices of $\phi$ in $\Sigma_1(\rd)\setminus{0}$ and that
different $\phi$ give rise to equivalent quasi-norms. (See Theorem
\ref{Thm:LebWindowTransf} in Section \ref{sec4}.)

\par

 Later on we need the following proposition.

 \par

 \begin{prop}\label{subsets}
Let $\Phi, \Phi _j$ be Young functions, $j=1,2$, $\omega_0 \in \mascP_E (\rd)$
and $\omega \in \mascP_E (\rdd )$. Then
$$
\mascS(\rd)\subseteq L^{\Phi} (\rd)\subseteq \mascS^{'}(\rd), \quad
\mascS(\rdd )\subseteq L^{\Phi _{1},\Phi _{2}}(\rdd )\subseteq \mascS^{'}(\rdd ),
$$
$$
\Sigma_1 (\rd) \subseteq L^{\Phi}_{(\omega_0)} (\rd)
\subseteq \Sigma_1 ' (\rd), \quad \Sigma_1 (\rdd ) \subseteq L^{\Phi _{1},\Phi _{2}}_{(\omega)}(\rdd )
\subseteq \Sigma_1 ' (\rdd ).
$$
\end{prop}

\par

\begin{proof}
Let $v_0 \in \mascP _E(\rd)$ and $v\in \mascP _E(\rr {2d})$ be chosen
such that $\omega_0$ is $v_0$-moderate and $\omega$ is $v$-moderate.
Since $L^{\Phi}_{(\omega_0)}(\rd)$ and $L^{\Phi _{1},\Phi _{2}}_{(\omega)}(\rdd )$
are invariant under translation and modulation, we have
$$
M^1_{(v_0)}(\rd)\subseteq
L^{\Phi}_{(\omega_0)}(\rd)\subseteq M^{\infty}_{(1/v_0)}(\rd),
$$
and
$$
M^1_{(v)}(\rd)\subseteq  L^{\Phi _{1},\Phi _{2}}_{(\omega)}(\rdd )
\subseteq M^{\infty}_{(1/v)}(\rd ).
$$
(see \cite[Theorem 12.1.8]{Gc2}). The result now follows from well-known
inclusions between modulation spaces, Schwartz spaces, Gelfand-Shilov
spaces, and their duals.
\end{proof}

\par

\subsection{Gabor frames}

\par

Let $E = \{ e_1,\dots ,e_d\}$ be an ordered basis of $\rr d$. Then the dual
dual basis $E'$ of $E$ is given by $E'=\{ \ep _1,\dots ,\ep _d\}$,
where $\scal {e_j}{\ep _k}= 2\pi \delta _{j,k}$.
The corresponding lattice $\Lambda =\Lambda _E$ and dual lattice $\Lambda '
=\Lambda _E'=\Lambda _{E'}$ are given by
\begin{align}
\Lambda = \Lambda _E
&=
\sets {n_1e_1+\cdots +n_de_d}{(n_1,\dots ,n_d)\in \zz d}
\label{Eq:LatticeDef}
\intertext{and}
\Lambda _E'
&=
\sets {\nu _1\ep _1+\cdots +\nu _d\ep _d}{(\nu _1,\dots ,\nu _d)\in \zz d}
\notag
\end{align}
respectively.

\par

\begin{defn}
Let $\omega, v \in \mascP_E (\rdd )$ be such that $\omega$ is
$v$-moderate, $\phi,\psi\in M^1_
{(v)}(\rd)$, $\ep >0$ and let $\Lambda \subseteq \rd$ be as in \eqref{Eq:LatticeDef}
with dual lattice $\Lambda' \subseteq \rd$.
\begin{enumerate}
    \item  The \emph{analysis operator} $C_{\phi}^{\ep , \Lambda}$
    is the operator from $M^\infty_{(\omega)}(\rd)$ to $\ell ^\infty _{(\omega)}
    (\ep \Lambda )\times (\ep \Lambda')$, given by
$$
C_{\phi}^{\ep , \Lambda} f
\equiv
\{ V_\phi f( j,\iota)\} _{(j,\iota)\in (\ep \Lambda )\times (\ep \Lambda')}.
$$
\item  The \emph{synthesis operator} $D_
\psi^{\ep , \Lambda}$ is the operator from $\ell^\infty_{(\omega)}((\ep \Lambda )
\times (\ep \Lambda'))$
to $M^\infty_{(\omega)}(\rd)$, given by
$$
D_{\psi}^{\ep , \Lambda} c
\equiv
\sum_{j\in \ep \Lambda}
\sum_{\iota \in \ep \Lambda'}
c(j,\iota)e^{i \scal \cdo \iota}\psi (\cdo- j).
$$
\item  The \emph{Gabor frame operator} $S_{\phi,\psi}^{\ep , \Lambda}$
is the operator on $M^\infty_{(\omega)}(\rd)$, given by
$D_{\psi}^{\ep , \Lambda} \circ C_{\phi}^{\ep , \Lambda},$ i.e.
$$
S_{\phi,\psi}^{\ep , \Lambda}
f \equiv
\sum_{j\in \ep \Lambda}
\sum_{\iota \in \ep \Lambda'}
V_\phi f(j,\iota)
e^{i\scal \cdo \iota}\psi (\cdo- j).
$$
\end{enumerate}
\end{defn}

\par

We set $C_{\phi}^{\ep }=C_{\phi}^{\ep _1, \Lambda} $,
$D_{\psi}^{\ep }=D_{\psi}^{\ep _1, \Lambda}$ and
$S_{\phi,\psi}^{\ep} = S_{\phi,\psi}^{\ep _1, \Lambda}$
when
$\Lambda=(2\pi)^\frac 12 \zd$ and $\ep _1=(2\pi)^{-\frac 12}\ep $.
It follows that
\begin{align*}
C_\phi ^{\ep } f
&=
\{V_\phi f(j,\iota)\}_{j,\iota \in \ep \zd}
\intertext{and}
D_\psi^{\ep }c
&=
\sum\limits_{j,\iota \in \ep \zd}c(j,\iota)
e^{i\scal \cdo \iota}\psi(\cdo-j).
\end{align*}

\par

The next result shows that it is possible to find suitable $\phi$ and
$\psi$ in the previous definition.

\par

\begin{lemma}\label{Lemma:GoodFrames0}
Let $\Lambda \subseteq \rd$ be as in \eqref{Eq:LatticeDef}
with dual lattice $\Lambda' \subseteq \rd$,
$v\in \mascP _E(\rr {2d})$ and $\phi \in M^1_{(v)}(\rr d) \setminus{0}$.
Then there is an $\ep >0$ and $\psi \in M^1_{(v)}(\rr d) \setminus{0}$ such that
\begin{equation}\label{Eq:GoodFrames0}
 \{ \phi (x-j)e^{i\scal x\iota } \} _{(j,\iota )\in \ep (\Lambda \times \Lambda ')}
\quad \text{and}\quad
 \{ \psi (x-j)e^{i\scal x\iota } \} _{(j,\iota )\in \ep (\Lambda \times \Lambda ')}
\end{equation}
are dual frames to each others.
\end{lemma}

\par

\begin{rem}\label{Remark:GoodFrames0}
There are several ways to achieve dual frames \eqref{Eq:GoodFrames0}.
In fact, let $v, v_0\in \mascP_E(\rdd )$ be submultiplicative such that $\omega$ is
$v$-moderate and $L^1_{(v_0)}(\rdd )\subseteq L^r(\rdd ),\ r\in(0,1]$. Then Lemma
\ref{Lemma:GoodFrames0} guarantees that for some choice of $\phi, \psi
\in M^1_{(v_0 v)}(\rd)\subseteq M^r_{(v)}(\rd)$ and lattice $\Lambda$
, the set in \eqref{Eq:GoodFrames0} are dual frames to each other, and that
$\psi = (S^\Lambda_{\phi,\phi})^{-1}\phi$. (Cf.
\cite[Proposition 1.5 and Remark 1.6]{Toft16}.)
\end{rem}

\par

\begin{lemma}\label{Lemma:GoodFrames}
Let $\Lambda \subseteq \rd$ be as in \eqref{Eq:LatticeDef}
with dual lattice $\Lambda' \subseteq \rd$,
$\phi _1,\phi _2 \in \Sigma_1(\rd)\setminus{0}$ and
\begin{equation*}
\varphi (x,\xi) = \phi _1(x)\overline{\widehat{\phi}_{2}(\xi)}e^{-i\scal {x}{\xi}}.
\end{equation*}
Then there is an $\ep >0$ such that
\begin{equation*}
\{\varphi(x-j,\xi-\iota)e^{i(\scal x \kappa+ \scal k \xi)} \} _{(j,\iota ),(k,\kappa )
\in \ep (\Lambda \times \Lambda ')}
\end{equation*}
is a Gabor frame with canonical dual frame
\begin{equation*}
\{\psi(x-j,\xi-\iota)e^{i(\scal x \kappa + \scal k \xi)}\} _{(j,\iota ),(k,\kappa )
\in \ep (\Lambda \times \Lambda ')}
\end{equation*}
where $\psi = (S_{\varphi,\varphi}^{\Lambda^2\times\Lambda^2})^{-1} \varphi$
belongs to $M_{(v)}^r(\rr d)$ for every $r>0$.
\end{lemma}

\par

The next result gives some information about the roles that $\Phi _1$ and
$\Phi _2$ play for $M^{\Phi _{1},\Phi _{2}}$ in the Banach space case. We
omit the proof since it can be found in \cite{SchF}.

\par

\begin{prop}\label{Prop:ModPhiPsi}
Let $\Lambda \subseteq \rr d$ be a lattice given by
\eqref{Eq:LatticeDef}, $\Phi _{j},\Psi _{j}$ be Young functions, $j=1,2$.
Then the following conditions are equivalent:
\begin{enumerate}
    \item $M^{\Phi _{1},\Phi _{2}}(\rd)\subseteq M^{\Psi _{1}, \Psi _{2}}(\rd)$;

    \vrum

    \item $\ell^{\Phi _{1},\Phi _{2}}(\Lambda \times \Lambda )
    \subseteq \ell^{\Psi _{1}, \Psi _{2}}(\Lambda \times \Lambda )$;

    \vrum

    \item there is a constant $t_0 >0$ such that $\Psi _{j} (t)\lesssim  \Phi _{j} (t)$
    for all $0\leq t\leq t_0$.
\end{enumerate}
\end{prop}

\par

\par

In section \ref{sec4} we extend Proposition \ref{Prop:ModPhiPsi} to quasi-Banach case.
(See Theorem \ref{Thm:OrliczModInvariance} and Proposition \ref{Prop:OrliczModInvariance}).

\par

\subsection{The Bargmann transform and modulation spaces}

\par

We finish the section by recalling the Bargmann transform and its mapping properties
on modulation spaces.

\par

The Bargmann kernel of dimension $d$ is given by
$$
\mathfrak{A}_d(z,y)= \pi^{-\frac d 4}\exp
\left(
-\frac 12 (\scal z z +|y|^2)+2^{\frac 12}\scal zy
\right), \quad y\in \rd,\ z\in \cc d.
$$
Here
$$
\scal zw=\sum\limits _{j=1}^d z_j w_j, \quad z=(z_1,\cdots ,z_d)\in \cc d,
\ \text{and} \
w=(w_1,\cdots,w_d)\in \cc d.
$$
(See \cite{Ba}.) It follows that $y\mapsto \mathfrak{A}_d(z,y)$
belongs to $\maclS_{1/2}(\rd)$ for every $z\in \cc d$.

\par

The Bargmann transform, $\mathfrak{V}_d f$ of $f\in \maclS'_{1/2}(\rd)$
is defined by
$$
(\mathfrak{V}_d f)(z)= \scal f{\mathfrak{A}_d(z,\cdo)}.
$$
There are several results on Bargmann images of well-known function
and distribution spaces. For example, it is proved already in \cite{Ba}
that $\mathfrak{V}_d$ is isometric bijection from $L^2(\rd)=M^2(\rd)$
into $L^2(d\mu)\cap A(\cc d)$. Here
$$
d\mu(z)=\pi^{-d}e^{-|z|^2}\, d\lambda (z),
$$
where $d\lambda (z)$ is the Lebesgue measure on $\cc d$, and
$A(\cc d)$ is the set of analytic functions on $\cc d$.
A more general result of the preceding result concerns
Proposition \ref{Prop:ModBargmannImages}
below which is a special case of \cite[Theorem 4.8]{Toft18}.
(See also \cite{FGW,SiT2,Toft10} for sub results.)

\par

If $F$ is a measurable function on $\cc d$ and $p,q\in (0,\infty ]$, then
\begin{align*}
\nm {F} {B^{p,q}_{(\omega)}}
&\equiv
\nm {F_{p,\omega}}{L^{q _2}(\rr d)},
\intertext[0ex]{where}
F_{p,\omega}(\xi )
&\equiv
(2\pi )^{-\frac d2}
e^{-\frac 12\cdot {|\xi |^2}}
\nm {e^{-\frac 14\cdot {|\cdo |^2}} F(2^{-\frac 12}(\cdo -i\xi ))
\omega (\cdo ,\xi )}{L^{p}(\rr d)}.
\end{align*}
We let $B^{p,q}_{(\omega)}(\cc d)$ be the set of all measurable
functions $F$ on $\cc d$ such that $\nm {F} {B^{p,q}_{(\omega)}}$ is finite.
We also let $A^{p,q}_{(\omega)}(\cc d) = B^{p,q}_{(\omega)}(\cc d)
\bigcap A(\cc d)$ with topology inherited from the topology in
$B^{p,q}_{(\omega)}(\cc d)$.

\par

\begin{prop}\label{Prop:ModBargmannImages}
Let $p,q\in (0,\infty]$ and $\omega$ be a weight on $\rr {2d}\simeq \cc d$.
Then $\mathfrak{V}_d$ from $\maclS_{1/2}(\rd)$ to $A(\cc d)$
is uniquely extendable to an isometric bijective map from $M^{p,q}_
{(\omega)}(\rd)$ to $A^{p,q}_{(\omega)}(\cc d)$.
\end{prop}

\par

Apart from Proposition \ref{Prop:ModBargmannImages}, there are several
characterizations of well-known
function and distribution spaces via their images under the Bargmann transform.
For example, convenient characterization of $\maclH _\flat (\rr d)$, $\maclS_s(\rd)$,
$\Sigma_s (\rd)$, $\mathscr{S}(\rd )$ and their duals can be found in
\cite{Toft10,Toft18} for $s>0$.
Especially we remark that the Bargmann transform on $L^2(\rd)$ is uniformly extendable
to a bijective map from $\maclH_\flat' (\rd)$ to $A(\cc d)$, and restricts to a
bijective map $\maclH_\flat (\rd)$ to the set of all entire functions $F$ on $\cc d$
such that $|F(z)| \lesssim e^{r|z|}$ for some $r>0$.

\par

Let $\Phi _1$ and $\Phi _2$ be quasi Young functions.
In a similar way as for the definition of $B^{p,q}_{(\omega)}(\cc d)$ above, we let
$B^{\Phi _1, \Phi _2}_{(\omega)}(\cc d)$ be the set of all measurable functions
$F$ on $\cc d$ such that $\nm {F} {B^{\Phi _1, \Phi _2}_{(\omega)}}$ is finite, where
\begin{align*}
\nm {F} {B^{\Phi _1, \Phi _2}_{(\omega)}}
&\equiv
\nm {F_{\Phi _1,\omega}}{L^{\Phi _2}(\rr d)},
\intertext[0ex]{with}
F_{\Phi _1,\omega}(\xi )
&\equiv
(2\pi )^{-\frac d2}
e^{-\frac 12\cdot {|\xi |^2}}
\nm {e^{-\frac 14 \cdot {|\cdo |^2}} F(2^{-\frac 12}(\cdo -i\xi ))
\omega (\cdo ,\xi )}{L^{\Phi _1}(\rr d)}.
\end{align*}
We also let $A^{\Phi _1,\Phi _2}_{(\omega)}(\cc d)=B^{\Phi _1,\Phi _2}_{(\omega)}
(\cc d)\cap A(\cc d)$ with topology inherited from the topology in
$B^{\Phi _1,\Phi _2}_{(\omega)}(\cc d)$.
It follows that $B^{p,q}_{(\omega)}={B^{\Phi _1,\Phi _2}_{(\omega)}}$
and $A^{p,q}_{(\omega)}={A^{\Phi _1,\Phi _2}_{(\omega)}}$
when $\Phi _1$ and $\Phi _2$ are chosen $\Phi _1(t)=\frac{t^p}{p}$
and $\Phi _2(t)=\frac{t^q}{q}$, giving that $L^{\Phi _1,\Phi _2}_{(\omega)}
=L^{p,q}_{(\omega)}$.

\par

\section{Continuity and Bargmann images of Orlicz
modulation spaces}\label{sec2}

\par

In this section we extend Proposition \ref{Prop:ModBargmannImages}
to more general weights and to the Orlicz case (see Theorem \ref{Thm:Bargmann}).
At the same time we prove that the Orlicz modulation spaces are quasi-Banach
spaces, by deducing similar facts of their Bargmann images.

\par

The extension of Proposition \ref{Prop:ModBargmannImages} is the following.
Here we let $\phi(x)=\pi^{-\frac d4}e^{-\frac 12\cdot {|x|^2}}$ in the modulation
space norms.

\par

\begin{thm} \label{Thm:Bargmann}
Let $\omega$ be a weight on $\rdd \simeq \cc d$ and $\Phi _1$, $\Phi _2$ be quasi
Young functions of order $r_0 \in(0,1]$. Then the following is true:
\begin{enumerate}
    \item $A^{\Phi _1,\Phi _2}_{(\omega)}(\cc d)$, $B^{\Phi _1,\Phi _2}_{(\omega)}(\cc d)$ and
    $M^{\Phi _1,\Phi _2}_{(\omega)}(\rd)$ are quasi-Banach spaces of order $r_0$;

\vrum

    \item the Bargmann transform is isometric and bijective from
    $M^{\Phi _1,\Phi _2}_{(\omega)}(\rd)$ to $A^{\Phi _1,\Phi _2}_{(\omega)}(\cc d)$.
\end{enumerate}
\end{thm}

\par

We need the following lemma for the proof. Here and in what follows we let
$\eabs z=(1+|z|^2)^{\frac 12}$ when $z\in \cc d$.

\par

\begin{lemma}\label{Lem.Bargmann}
Let $\omega$ be a weight on $\rdd \simeq \cc d$, $\Phi _j$ be quasi
Young functions of order $r_0 \in(0,1]$, $j=1,2$. Then the following is true:
\begin{enumerate}
    \item if $\rho \in (0,1]$ and $\omega_r(z)=\omega(z)e^{-r|z|^\rho}$ for
    some $r>0$, then $A^{\Phi _1,\Phi _2}_{(\omega)}(\cc d)$ is continuously
    embedded in $A^{r_0}_{(\omega_r)}(\cc d)$;

    \vrum

    \item if $r> \frac{2d}{r_0}$ and $\omega_r(z)=\omega(z)\eabs z^{-r}$, then
    $A^{\Phi _1,\Phi _2}_{(\omega)}(\cc d)$ is continuously embedded in
    $A^{r_0}_{(\omega_r)}(\cc d)$.
\end{enumerate}
\end{lemma}

\par

\begin{proof}
We have $e^{-r|z|^\rho}\lesssim \eabs z ^{-r}$, which implies that
it suffices to prove (2). Since  $\Phi _1$ and $\Phi _2$ are quasi Young
functions of order $r_0$, we have
\begin{align*}
    \Phi _j(t)\ge
    \begin{cases}
      0, & 0\le t\le t_0,
      \\[1ex]
      C (t-t_0)^{r_0}, &t>t_0
    \end{cases}
\end{align*}
for some choices of $t_0>0$ and $C>0$. This implies that $L^{\Phi _1,\Phi _2}_{(\omega)}
\subseteq L^{r_0}_{(\omega)}+L^{\infty}_{(\omega)}$. Since
$L^{\infty}_{(\omega)} \subseteq L^{r_0}_{(\omega_r)}$ and
$L^{r_0}_{(\omega)} \subseteq L^{r_0}_{(\omega_r)}$ we get
$L^{\Phi _1,\Phi _2}_{(\omega)} \subseteq L^{r_0}_{(\omega_r)}$,
which in turn leads to $A^{\Phi _1,\Phi _2}_{(\omega)}(\cc d)
\subseteq A^{r_0}_{(\omega_r)}(\cc d)$, with continuous inclusion.
\end{proof}

\par

\begin{proof}[Proof of Theorem \ref{Thm:Bargmann}]
Let $\rho \in (0,1)$ and $\omega_r=\omega\cdot e^{-r|\cdo |^\rho}$. Since
$B^{\Phi _1,\Phi _2}_{(\omega)}(\cc d)$ is essentially a weighted Orlicz
space, the completeness of $B^{\Phi _1,\Phi _2}_{(\omega)}(\cc d)$ follows from
the completeness of $L^{\Phi _1,\Phi _2}_{(\omega)}(\rdd )$. Suppose that
$\{F_j\}_{j=1}^\infty$ is a Cauchy sequence in $A^{\Phi _1,\Phi _2}_{(\omega)}(\cc d)$. Since
$A^{\Phi _1,\Phi _2}_{(\omega)}(\cc d)$ is continuously embedded in $A^{r_0}_{(\omega_r)}(\cc d)$
for every $r>0$, $\{F_j\}_{j=1}^\infty$ is a Cauchy sequence in $A^{r_0}_{(\omega_r)}(\cc d)$ as well.

\par

By completeness of $B^{\Phi _1,\Phi _2}_{(\omega)}(\cc d)$ and  $A^{r_0}_{(\omega_r)}(\cc d)$,
there are $F \in B^{\Phi _1,\Phi _2}_{(\omega)}(\cc d)$ and $F_0 \in A^{r_0}_{(\omega_r)}(\cc d)$
such that
$$
F_j \rightarrow F \; \text{in} \; B^{\Phi _1,\Phi _2}_{(\omega)}(\cc d) \quad
\textstyle{and}\quad F_j \rightarrow F_0 \; \text{in}\; A^{r_0}_{(\omega_r)}(\cc d).
$$
Since $B^{\Phi _1,\Phi _2}_{(\omega)}(\cc d)$ and $ A^{r_0}_{(\omega_{r})}(\cc d)$
are equipped with weighted Lebesgue norms we have $F=F_0$ a.e. Since $F_0 \in A(\cc d)$,
we get
$$
F\in B^{\Phi _1,\Phi _2}_{(\omega)}(\cc d) \cap
A(\cc d)= A^{\Phi _1,\Phi _2}_{(\omega)}(\cc d)
$$
giving the completeness of $A^{\Phi _1,\Phi _2}_{(\omega)}(\cc d)$.

\par

By the completeness of  $A^{\Phi _1,\Phi _2}_{(\omega)}(\cc d)$, the
completeness of  $M^{\Phi _1,\Phi _2}_{(\omega)}(\rd)$ follows if we prove (2).

\par

By the definitions it follows that
\begin{equation} \label{Eq:Bargmannnorm}
    \nm {\mathfrak{V}_d f}{B^{\Phi _1,\Phi _2}_{(\omega)}}
    =\nm {f} {M^{\Phi _1,\Phi _2}_{(\omega)}}, \quad f\in \maclH_\flat(\rd).
\end{equation}
Since $\mathfrak{V}_d f\in A(\cc d)$ when $f\in \maclH_\flat(\rd)$, it follows from
\eqref{Eq:Bargmannnorm} that $\mathfrak{V}_d f\in B^{\Phi _1,\Phi _2}_{(\omega)}
\cap A(\cc d)=A^{\Phi _1,\Phi _2}_{(\omega)}(\cc d)$ when
$f\in M^{\Phi _1,\Phi _2}_{(\omega)}(\rd)$. This shows that $\mathfrak{V}_d$ is an
isometric injection from $M^{\Phi _1,\Phi _2}_{(\omega)}(\rd)$ to
$A^{\Phi _1,\Phi _2}_{(\omega)}(\cc d)$. We need to prove the surjectivity of this map.

\par

Suppose $F\in A^{\Phi _1,\Phi _2}_{(\omega)}(\cc d)\subseteq A(\cc d)$. Since any element
in $ A(\cc d)$ is a Bargmann transform of an element in $\maclH'_\flat(\rd)$, we have
$$
F=\mathfrak{V}_d f
$$
for some $f\in \maclH'_\flat(\rd)$. By \eqref{Eq:Bargmannnorm} we get
$$
\nm {f}{M^{\Phi _1,\Phi _2}_{(\omega)}}
=  \nm {\mathfrak{V}_d f}{A^{\Phi _1,\Phi _2}_{(\omega)}}
<\infty ,
$$
giving that $f\in M^{\Phi _1,\Phi _2}_{(\omega)}(\rd)$. This gives the asserted surjectivity,
and thereby the result.
\end{proof}

\par

We have now the following inclusion relations between Orlicz modulation spaces,
Gelfand-Shilov and Schwartz spaces and their duals. We recall Subsection
\ref{subsec1.2} for notations on weight classes.

\par

\begin{prop}\label{Prop:ThreeCases}
Let $s_0,\sigma _0\ge \frac 12$, $s,\sigma > \frac 12$, $\Phi _1$, $\Phi _2$
be quasi Young functions of order $r_0 \in (0,1]$,
$$
v_{r}(x,\xi )= (1+|x|+|\xi |)^r
\quad \text{and}\quad
v_{s,\sigma ,r}(x,\xi )= e^{r(|x|^{\frac 1s}+|\xi |^{\frac 1\sigma})}.
$$
Then
\begin{alignat*}{2}
\Sigma _s^\sigma (\rr d) &= \bigcap _{r>0}M^{\Phi _1,\Phi _2}_{(v_{s,\sigma ,r})}(\rd ), &
\quad
(\Sigma _s^\sigma )'(\rr d) &= \bigcup _{r>0}M^{\Phi _1,\Phi _2}_{(1/v_{s,\sigma ,r})}(\rd ),
\\[1ex]
\maclS _{s_0}^{\sigma _0}(\rr d) &= \bigcup _{r>0}
M^{\Phi _1,\Phi _2}_{(v_{s_0,\sigma _0,r})}(\rd ), &
\quad
(\maclS _{s_0}^{\sigma _0} )'(\rr d) &= \bigcap _{r>0}
M^{\Phi _1,\Phi _2}_{(1/v_{s_0,\sigma _0,r})}(\rd ),
\intertext{and}
\mascS (\rr d) &= \bigcap _{r>0}M^{\Phi _1,\Phi _2}_{(v_{r})}(\rd ), &
\quad
\mascS '(\rr d) &= \bigcup _{r>0}M^{\Phi _1,\Phi _2}_{(1/v_{r})}(\rd ).
\end{alignat*}
\end{prop}

\par

\begin{cor}\label{Cor:ThreeCases}
Let $\Phi _1$, $\Phi _2$ be the same as in Proposition
\ref{Prop:ThreeCases}. Then
\begin{alignat*}{3}
    \Sigma_1 &= \bigcap _{\omega \in \mascP _E}
    M^{\Phi _1,\Phi _2}_{(\omega)}, &
    \quad
    \maclS _1 &= \bigcap _{\omega \in \mascP _E^0}
    M^{\Phi _1,\Phi _2}_{(\omega)}, &
    \quad
    \mascS &= \bigcap _{\omega \in \mascP}
    M^{\Phi _1,\Phi _2}_{(\omega)}
\intertext{and}
    \Sigma_1' &=
    \bigcup _{\omega \in \mascP _E}
    M^{\Phi _1,\Phi _2}_{(\omega)}, &
    \quad
    \maclS _1' &= \bigcup _{\omega \in \mascP _E^0}
    M^{\Phi _1,\Phi _2}_{(\omega)}, &
    \quad
    \mascS ' &= \bigcup _{\omega \in \mascP}
    M^{\Phi _1,\Phi _2}_{(\omega)}.
\end{alignat*}
\end{cor}

\par

\begin{proof}
The result follows from the definitions, Proposition \ref{Prop:ThreeCases}
and the fact that
$$
M^{\Phi _1,\Phi _2}_{(v)}(\rd)\subseteq M^{\Phi _1,\Phi _2}_{(\omega)}(\rd)
\subseteq M^{\Phi _1,\Phi _2}_{(1/v)}(\rd),
$$
when $\omega ,v\in \mascP _E(\rr {2d})$ are such that $v$ is submultiplicative and $\omega$
is $v$-moderate.
\end{proof}

\par

We need the following lemma for the proof of Proposition \ref{Prop:ThreeCases}.

\par

\begin{lemma}\label{Lemma:Embeddings}
Let $\Phi _1,\Phi _2$ be quasi-Young functions of order $r_0\in(0,1]$ and let
$\omega$ be a weight on $\rdd$. Then
$$
L^{r_0}_{(\omega)}(\rdd )\cap L^{\infty}_{(\omega)}(\rdd )
\hookrightarrow
L^{\Phi _1,\Phi _2}_{(\omega)}(\rdd )
\hookrightarrow
L^{r_0}_{(\omega)}(\rdd ) + L^{\infty}_{(\omega)}(\rdd ).
$$
\end{lemma}

\par

A proof of Lemma \ref{Lemma:Embeddings} is essentially given in \cite{SchF}.
In order to be self-contained we here give the arguments.

\par

\begin{proof}
By the mapping $f \mapsto |f \cdot \omega|^{r_0}$, we reduce
ourselves to the case $\omega =1$ and $r_0 = 1$.
Since $\Phi _1$ and $\Phi _2$ are convex, there are
constants $t_1, t_2 >0$ and $C_1,C_2>0$ such that if
\begin{align*}
    \Psi _1(t)=
    \begin{cases}
      0, & 0\le t\le t_1,
      \\[1ex]
      C_1 (t-t_1), &t>t_1
    \end{cases}
\end{align*}
and
\begin{align*}
    \Psi _2(t)=
    \begin{cases}
      C_2 t, & 0\le t\le t_2,
      \\[1ex]
      \infty, &t>t_2,
    \end{cases}
\end{align*}
then
$$
\Psi _1(t)\le \Phi _j(t) \le \Psi _2(t), \quad j=1,2.
$$
This gives
\begin{multline*}
L^1(\rdd )\cap L^{\infty}(\rdd ) = L^{\Psi _2}(\rdd )
\hookrightarrow L^{\Phi _1,\Phi _2}(\rdd )
\\[1ex]
\hookrightarrow L^{\Psi _1}(\rdd ) = L^1(\rdd ) + L^{\infty}(\rdd ). \qedhere
\end{multline*}
\end{proof}

\par

\begin{cor}\label{Cor:ModWienerInv}
Let $\Phi _1,\Phi _2$ be quasi-Young functions of order $r_0\in(0,1]$ and
let $\omega\in \mascP_E(\rdd )$. Then
$$
M^{r_0}_{(\omega)}(\rd)
\hookrightarrow M^{\Phi _1,\Phi _2}_{(\omega)}(\rd)
\hookrightarrow M^{\infty}_{(\omega)}(\rd).
$$
\end{cor}

\par

\begin{proof}
Since $\omega\in \mascP_E(\rdd )$, it follows that $M^{p,q}_{(\omega)}(\rd)$
increases with $p,q\in(0,\infty]$. Hence Lemma \ref{Lemma:Embeddings} gives
\begin{multline*}
M^{r_0}_{(\omega)}(\rd)
= M^{r_0}_{(\omega)}(\rd)\cap M^{\infty}_{(\omega)}(\rd)
\hookrightarrow M^{\Phi _1,\Phi _2}_{(\omega)}(\rd)
\\[1ex]
\hookrightarrow M^{r_0}_{(\omega)}(\rd) + M^{\infty}_{(\omega)}(\rd)
= M^{\infty}_{(\omega)}(\rd).\qedhere
\end{multline*}
\end{proof}

\par

\begin{proof}[Proof of Proposition \ref{Prop:ThreeCases}]
If $\Phi _j$ are chosen such that $M^{\Phi _1,\Phi _2}_{(\omega )} =
M^{p,q}_{(\omega )}$ for some $p,q\in (0,\infty ]$, then the result follows
by a combination of Remark 1.3 (6) in \cite{Toft8}, Theorem 3.9 in \cite{Toft10},
and Proposition 6.5 in \cite{Toft18}. The result now follows for general
$\Phi _j$ by combining Theorem 3.2 in \cite{Toft10} with Lemma
\ref{Lemma:Embeddings}.
\end{proof}


\par

\section{Convolution estimates
for quasi-Orlicz spaces}\label{sec3}

\par

In this section we extend the convolution estimates in \cite{GaSa} for Lebesgue spaces
to the case of quasi Orlicz spaces. In the first part we deduce discrete convolution
estimates between elements in discrete Orlicz and Lebesgue spaces. Thereafter
we focus on the semi-continuous convolution, and prove corresponding estimates
for $L^{\Phi}_{(\omega)}(\rdd )$ or in convolutions between elements in
$L^{\Phi _1,\Phi _2}_{(\omega)}(\rdd )$, and $\ell^{r_0}_{(v)}$.
In the end we also deduce similar estimates for continuous convolutions after
$L^{\Phi}_{(\omega)}(\rdd )$, $L^{\Phi _1,\Phi _2}_{(\omega)}(\rdd )$ and
$\ell^{r_0}_{(v)}$ are replaced
by the Wiener spaces $W(L^{\Phi _1,\Phi _2}_{(\omega)})$ and $W(L^1,L^{r_0}_{(v)})$.

\par

Our investigations involve weighted Orlicz spaces where the weights should satisfy
conditions of the form
\begin{alignat}{2}
\omega _0(x+y)&\lesssim \omega _1(x)\omega _2(y),&\quad
x,y &\in \rr d
\label{Eq:StandardWeightCond1}
\intertext{or}
\omega _0(x+y,\xi +\eta )&\lesssim \omega _1(x,\xi )\omega _2(y,\eta ),&\quad
x,y,\xi ,\eta  &\in \rr d.
\label{Eq:StandardWeightCond2}
\end{alignat}

\subsection{Discrete convolution estimates on
discrete Orlicz spaces.}

\par


\par

We recall that the discrete convolution
between $a\in \ell _0'(\Lambda )$ and $b\in \ell _0(\Lambda )$ is defined by
$$
(a*b)(n)=\sum \limits _{k\in \Lambda} a(k)b(n-k).
$$

\par

\begin{lemma}\label{Lemma:DiscYoung}
Let $\Lambda$ be as in \eqref{Eq:LatticeDef}, $\omega _0 \in \mascP_E (\rd)$,
$\maclB _1\subseteq \ell _0'(\Lambda )$ be a quasi-Banach space, and let
$\maclB _2\subseteq \ell _0'(\Lambda )$ be a quasi-Banach space
of order $r_0\in (0,1]$ such that $k\mapsto a(k-n)\in \maclB _2$
when $a\in \maclB _1$, $n\in \Lambda$ and
$$
\nm {a(\cdo -n)}{\maclB _2}\le C \omega _0(n)\nm {a}{\maclB _1},\qquad
a\in \maclB _1,\ n\in \Lambda ,
$$
for some constant $C>0$.
Then the map $(a,b) \mapsto a*b$ from
$\ell_0 (\Lambda ) \times \maclB _1$ to $\ell_0 ' (\Lambda )$
is uniquely extendable to a continuous map from
$\ell ^{r_0}_{(\omega _0)}(\Lambda ) \times \maclB _1$ to
$\maclB _2$, and
\begin{equation} \label{Eq:ConvYoung}
\nm {a*b}{\maclB _2}
\le
C\nm a{\ell^{r_0}_{(\omega _0)}} \nm b{\maclB _1},
\qquad a\in \ell^{r_0}_{(\omega _0)}(\Lambda ),\ b\in \maclB _1.
\end{equation}
\end{lemma}

\par

Lemma \ref{Lemma:DiscYoung} follows by similar arguments as
\cite[Lemma 2.5]{GaSa}. In order to be self-contained we here show the arguments.

\par

\begin{proof}
Since $\ell _0$ is dense in $\ell^{r_0}_{(\omega _0)}$, the result follows if
we prove \eqref{Eq:ConvYoung} for $a\in \ell _0(\Lambda )$.

\par

Since $\maclB _j$ are $r_0$-Banach spaces, $j=1,2$, we have
\begin{multline*}
\nm {a*b}{\maclB _2}^{r_0} = \Nm {\sum _{k\in \Lambda} a(k)b(\cdo -k)}{\maclB _2}^{r_0}
\le
\sum _{k\in \Lambda} \nm {a(k)b(\cdo -k)}{\maclB _2}^{r_0}
\\[1ex]
\le
C^{r_0}\sum _{k\in \Lambda} (|a(k)\omega _0(k))^{r_0}|\nm {b}{\maclB _1}^{r_0}
=
(C\nm a{\ell^{r_0}_{(\omega _0)}} \nm b{\maclB _1})^{r_0},
\end{multline*}
which gives \eqref{Eq:ConvYoung}.
\end{proof}

\par

By choosing $\maclB _j$ as $\ell^{\Phi} _{(\omega _j)}(\zd)$
or $\ell^{\Phi _1,\Phi _2} _{(\omega _j)}(\zz {2d})$ in the previous lemma,
for suitable $\omega _j$, we deduce the following.

\par

\begin{cor}\label{Cor:Young}
Let $\Lambda$ be as in \eqref{Eq:LatticeDef}, $\Phi$, $\Phi _1$ and
$\Phi _2$ be quasi-Young functions of order $r_0\in(0,1]$.
Then the following is true:
\begin{enumerate}
\item suppose that
$\omega _j\in \mascP_E (\rd )$, $j=0,1,2$ satisfy \eqref{Eq:StandardWeightCond1}.
Then the map $(a,b) \mapsto a*b$ from
$\ell_0 (\Lambda ) \times \ell^{\Phi} _{(\omega _2)}(\Lambda )$ to $\ell_0 ' (\Lambda )$
is uniquely extendable to a continuous map from
$\ell ^{r_0}_{(\omega _1)}(\Lambda ) \times \ell^{\Phi}_{(\omega _2)} (\Lambda )$ to
$\ell^{\Phi}_{(\omega _0)}(\Lambda )$;

\vrum

\item let $\Lambda ^2=\Lambda \times \Lambda$ and suppose that
$\omega _j\in \mascP_E (\rr {2d} )$, $j=0,1,2$ satisfy \eqref{Eq:StandardWeightCond2}.
Then the map $(a,b) \mapsto a*b$ from
$\ell_0 (\Lambda ^2) \times \ell^{\Phi _1,\Phi _2} _{(\omega _2)}(\Lambda ^2)$
to $\ell_0 ' (\Lambda ^2)$
is uniquely extendable to a continuous map from
$\ell ^{r_0}_{(\omega _1)}(\Lambda ^2) \times
\ell^{\Phi _1,\Phi _2}_{(\omega _2)} (\Lambda ^2)$ to
$\ell^{\Phi _1,\Phi _2}_{(\omega _0)}(\Lambda ^2)$.
\end{enumerate}
\end{cor}

\par

Next we perform similar investigations for semi-discrete convolutions.

\par

\begin{defn}
Let $\Lambda$ be as in \eqref{Eq:LatticeDef}.
The \emph{semi-discrete convolution} of
$a\in \ell_0(\Lambda)$ and $f\in \Sigma_1 '(\rd)$
with respect to $\Lambda$
is given by
$$
(a*_{\Lambda} f)(x)=\sum_{k\in \Lambda}a(k)f(x-k).
$$
\end{defn}

\par

For $\ep >0$ we also set $*_{\ep}=*_{\Lambda}$ when
$\Lambda=\ep \zd$. Then
$$
(a*_{\ep}  f)(x) =\sum_{k\in \ep \zd} a(k)f(x-k).
$$

\par

The following result corresponds to Lemma \ref{Lemma:DiscYoung} in the framework
of semi-discrete convolutions. Here and in what follows we let $\maclM (\rr d)$ be the
set of all (complex-valued) Borel measurable functions on $\rr d$.

\par

\begin{lemma}\label{Lemma:SemiDiscYoung}
Let $\Lambda$ be as in \eqref{Eq:LatticeDef}, $\omega _0 \in \mascP_E (\rd )$,
$\maclB _1\subseteq \maclM (\rr d)$ be a quasi-Banach space, and let
$\maclB _2\subseteq \maclM (\rr d)$ be a quasi-Banach space
of order $r_0\in (0,1]$ such that $y\mapsto f(y-x)\in \maclB _2$
when $f\in \maclB _1$, $x\in \rr d$ and
$$
\nm {f(\cdo -x)}{\maclB _2}\le  C\omega _0(x)\nm {f}{\maclB _1},\qquad
f\in \maclB _1,\ x\in \rr d,
$$
for some constant $C>0$.
Then the map $(a,f) \mapsto a*_\Lambda f$ from
$\ell_0 (\Lambda ) \times \maclB _1$ to $\maclM (\rr d)$
is uniquely extendable to a continuous map from
$\ell ^{r_0}_{(\omega _0)}(\Lambda ) \times \maclB _1$ to
$\maclB _2$, and
\begin{equation} \label{Eq:SemiDiscYoung}
\nm {a*_{\Lambda} f}{\maclB _2}
\leq
C\nm a{\ell^{r_0}_{(\omega _0)}(\Lambda )} \nm f{\maclB _1},
\qquad a\in \ell^{r_0}_{(\omega _0)}(\Lambda ),\ f\in \maclB _1.
\end{equation}
\end{lemma}

\par

\begin{proof}
We shall argue as in the proof of Lemma \ref{Lemma:DiscYoung}.
Since $\ell _0$ is dense in $\ell^{r_0}_{(\omega _0)}$, the result follows if
we prove \eqref{Eq:SemiDiscYoung} for $a\in \ell _0(\Lambda )$.

\par

Since $\maclB$ is an $r_0$-Banach space we have
\begin{multline*}
\nm {a*_{\Lambda}f}{\maclB _2}^{r_0} = \Nm {\sum _{k\in \Lambda} a(k)f(\cdo -k)}{\maclB _2}^{r_0}
\le
\sum _{k\in \Lambda} \nm {a(k)f(\cdo -k)}{\maclB _2}^{r_0}
\\[1ex]
\le
C^{r_0}\sum _{k\in \Lambda} (|a(k)\omega _0(k))^{r_0}|\nm {f}{\maclB _1}^{r_0}
=
(C\nm a{\ell^{r_0}_{(\omega _0)}} \nm f{\maclB _1})^{r_0}.
\qedhere
\end{multline*}
\end{proof}

\par

By choosing $\maclB _j$ as $L^{\Phi} _{(\omega _j)}(\rr d)$
or as $L^{\Phi _1,\Phi _2} _{(\omega _j)}(\rr {2d})$ in the previous lemma,
we deduce the following.

\par

\begin{cor}\label{Cor:SemiDiscYoung}
Let $\Lambda$ be as in \eqref{Eq:LatticeDef}, $\Lambda ^2=\Lambda \times \Lambda $,
$\Phi$, $\Phi _1$ and $\Phi _2$ be quasi-Young functions
of order $r_0\in(0,1]$. Then the following is true:
\begin{enumerate}
\item suppose that
$\omega _j\in \mascP_E (\rr {d} )$, $j=0,1,2$,
satisfy \eqref{Eq:StandardWeightCond1}.
Then the map $(a,f) \mapsto a*_\Lambda f$ from
$\ell_0 (\Lambda ) \times L^{\Phi} _{(\omega _2)}(\rr d)$ to $\maclM (\rr d)$
is uniquely extendable to a continuous map from
$\ell ^{r_0}_{(\omega _1)}(\Lambda ) \times L^{\Phi}_{(\omega _2)} (\rr d)$ to
$L^{\Phi}_{(\omega _0)}(\rr d)$;

\vrum

\item suppose that
$\omega _j\in \mascP_E (\rr {2d} )$, $j=0,1,2$,
satisfy \eqref{Eq:StandardWeightCond2}.
Then the map $(a,f) \mapsto a*_{\Lambda ^2}f$ from
$\ell_0 (\Lambda ^2) \times
L^{\Phi _1,\Phi _2} _{(\omega _2)}(\rr {2d})$
to $\maclM (\rr {2d})$
is uniquely extendable to a continuous map from
$\ell ^{r_0}_{(\omega _1)}(\Lambda ^2) \times
L^{\Phi _1,\Phi _2}_{(\omega _2)} (\rr {2d})$ to
$L^{\Phi _1,\Phi _2}_{(\omega _0)}(\rr {2d})$.
\end{enumerate}
\end{cor}

\par

In what follows we set $\chi=\chi_{[0,1]^{2d}}$ and $Q_d=[0,1]^{d}$.

\par

\begin{defn}\label{DefAmalgam}
Let $\Phi , \Phi _1, \Phi _2$ be quasi-Young
functions of order $r_0\in(0,1]$, and let $\omega \in \mascP_E (\rdd )$.
Then the Wiener-type space $W(L^{\Phi },L^{\Phi _1,\Phi _2}_{(\omega)})$ consists of
all measurable functions $F$ on $\rdd $, such that
\begin{multline*}
a_{F,\omega,\Phi }(k,\kappa)=
\Vert F(\cdo +(k,\kappa))\omega(\cdo +(k,\kappa))\Vert_{L^{\Phi }(Q_{2d})}
\\[1ex]
=
\Vert F\cdot \omega\cdot T_{(k,\kappa)}\chi\Vert_{L^{\Phi }(\rdd )},\quad k,\kappa\in \zd,
\end{multline*}
belongs to $\ell^{\Phi _1,\Phi _2}(\zdd)$. The (quasi-) norm on
$W(L^{\Phi },L^{\Phi _1,\Phi _2}_{(\omega)})$ is given by
$$
\Vert F\Vert_{W(L^{\Phi },L^{\Phi _1,\Phi _2}_{(\omega)})}
\equiv
\Vert a_{F,\omega,\Phi}\Vert_{\ell^{\Phi _1, \Phi _2}}.
$$
\end{defn}

\par

For conveniency we set $W(L^{\Phi _1,\Phi _2})=W(L^\infty ,L^{\Phi _1 ,\Phi _2})$
and $a_{F,\omega}=a_{F,\omega,\Phi }$ when $L^{\Phi}=L^\infty$. It is obvious that
\begin{equation} \label{Eq:Wiener}
    \nm{F}{L^{\Phi _1,\Phi _2}_{(\omega)}}
    \le
    \nm {F}{W(L^{\Phi _1,\Phi _2}_{(\omega)})}
\end{equation}
for every measurable function $F$ on $\rdd $.

\par

\begin{rem}\label{Rem:Gfunctional}
Let $r_0, \Phi _1, \Phi _2$ and $\omega$ be the same as in Definition
\ref{DefAmalgam}. Then the following is true:
\begin{enumerate}
    \item $\nm {F}{W(L^{\Phi _1,\Phi _2}_{(\omega)})}
    = \nm{G_{F,\omega}}{L^{\Phi _1,\Phi _2}}$, where
    $$
    G_{F,\omega}(x,\xi)=\sum\limits_{k,\kappa\in \zd}
    a_{F,\omega}(k,\kappa)\chi_{(k,\kappa)+Q_{2d}}(x,\xi).
    $$

    \vrum

    \item $\nm {F}{W(L^{\Phi _1,\Phi _2}_{(\omega)})}
    \asymp \nm {a_F}{\ell^{\Phi _1,\Phi _2}_{(\omega)}}
    =\nm{G_F}{L^{\Phi _1,\Phi _2}_{(\omega)}}$, where
    $$
    a_F =a_{F,1}\ \text{and}\ G_F=G_{F,1}.
    $$
\end{enumerate}
\end{rem}

\par

The next lemma corresponds to \cite[Lemma 3.8]{GaSa}
and give suitable Orlicz estimates of samples in terms of
Wiener norm estimates. Here and in what follows we let
$[t]$ be the integer part of $t\in \mathbf R$.

\par

\begin{lemma}\label{Lemma:No2.8}
Let $\Phi _1,\Phi _2$ be quasi Young functions of order $r_0 \in(0,1]$,
$\omega \in \mascP_E (\rdd )$, $F\in W(L^{\Phi _1,\Phi _2}_{(\omega)})$
be continuous, $c_F(k,\kappa)=F(\alpha k,\beta\kappa)$
for all $\alpha,\beta>0$ and let $C_\alpha = ([\frac{1}{\alpha}]+1)^{d}$.
Then $c_F\in\ell^{\Phi _1,\Phi _2}_{(\omega)}(\zdd)$ and for some
constant $C_\omega$ which only depends on $\omega$, it holds
\begin{equation} \label{rest}
  \Vert c_F\Vert_{\ell^{\Phi _1,\Phi _2}_{(\omega)}}
  \le
 C_\omega (C_\alpha C_\beta )^{\frac 1{r_0}} \Vert F \Vert_{W(L^{\Phi _1,\Phi _2}_{(\omega)})}.
\end{equation}
\end{lemma}

\par

\begin{proof}
The map $F\mapsto F\cdot\omega$ and the fact that $\omega$ is $v$-moderate
for some $v\in \mascP _E(\rr {2d})$, carry over the estimate
\eqref{rest} into the case $\omega = 1$.
Hence if suffices to prove the result for $\omega=1$. Let $a_F$
be the same as in Definition \ref{DefAmalgam}.
For $(\alpha k,\beta \kappa)\in (j, \iota)+Q_{2d}$ and $(j, \iota)\in \zdd$ we have
\begin{equation*}
  |c_F(k,\kappa)| \leq \Vert F(\cdo +(j,\iota))\Vert_{L^\infty(Q_{2d})}.
\end{equation*}
Since there are at most $C_{\alpha}$
 points $\alpha k\in j+Q_{d}$, the $L^{\Phi _1}$ norm over $k$
 is bounded by
 \begin{equation*}
   \sum_{k\in\zd}
   \Phi _{0,1}
   \left(
   \frac{|c_F(k,\kappa)|^{r_0}}{\lambda^{r_0}}
   \right)
   \leq
   C_{\alpha}
   \sum_{j\in\zd}
   \Phi _{0,1}
   \left(
   \frac{\Vert c_F(\cdo+(j,\iota)\Vert^{r_0}_{L^\infty}}
   {\lambda^{r_0}}
   \right).
 \end{equation*}
 By the definition of $\ell^{\Phi _1}$ norm,
 \begin{equation*}
\Vert c_F(\cdo,\kappa) \Vert^{r_0}_{\ell^{\Phi _1}}
\leq
C_{\alpha} \Vert a_F(\cdot,\iota) \Vert^{r_0}_{\ell^{\Phi _1}}
 \end{equation*}
for  $\beta \kappa\in \iota+Q_d$.

\par

Since there are at most
 $C_{\beta}=([\frac{1}{\beta}]+1)^{d}$ of $\kappa \in \zz d$ such that
 $\beta \kappa\in \iota+Q_d$, we get
 \begin{equation*}
  \sum_{\kappa\in\zd }
  \Phi _{0,2}
  \left(
  \frac{\|c_F(\cdo,\kappa)\|^{r_0}_{\ell^{\Phi _1}}}{\lambda^{r_0}}
  \right)
  \leq
  C_\beta \sum_{\iota\in\zd}
  \Phi _{0,2}
  \left(
  \frac{C_{\alpha} \|a_F(\cdo,\iota)\|^{r_0}_{\ell^{\Phi _1}}}{\lambda^{r_0}}
  \right).
 \end{equation*}
 By definition of $\ell^{\Phi _2}$ norm we get \eqref{rest}.
 \end{proof}

\par


\par

\begin{lemma}\label{Lemma:WienerConvEst}
Let $\Phi _1,\Phi _2$ be quasi Young functions of order $r_0 \in(0,1]$, and
$\omega_j\in \mascP_E (\rdd )$, $j=0,1,2$, be such that
\eqref{Eq:StandardWeightCond2} holds.
Then the map $(F,G) \mapsto F*G$ from $\Sigma_1(\rdd )\times\Sigma_1(\rdd )$
to $\Sigma_1(\rdd )$ extends uniquely to a continuous map from
$W(L^1,L^{r_0}_{(\omega _1)})\times W(L^{\Phi _1,\Phi _2}_{(\omega _2)})$ to
$ W(L^{\Phi _1,\Phi _2}_{(\omega _0)})$, and
\begin{equation} \label{2.9}
  \Vert F*G \Vert_{W(L^{\Phi _1,\Phi _2}_{(\omega _0)})}
  \lesssim\Vert F\Vert_{W(L^{1},L^{r_0}_{(\omega _1)})}
  \Vert G \Vert_{W(L^{\Phi _1,\Phi _2}_{(\omega _2)})}.
\end{equation}
\end{lemma}

\par

Lemma \ref{Lemma:WienerConvEst} is similar to \cite[Theorem 5.1]{Rau1}
when $Y$ in \cite{Rau1} is an Orlicz space. Lemma \ref{Lemma:WienerConvEst}
also essentially generalizes parts of \cite[Lemma 2.9]{GaSa}.

\par

\begin{proof}
We have $W(L^{\Phi _1,\Phi _2}_{(\omega _j)}) \subseteq L^\infty_{(\omega _j)}(\rdd )$
and that $\Sigma_1 (\rdd )$ is dense in $W(L^1, L^{r_0}_{(\omega _1)})$. Since $F*G$ is
uniquely defined when $F\in \Sigma_1 (\rdd )$ and $G\in L^\infty_{(\omega _2)}(\rdd )$
the result follows if we prove that \eqref{2.9} holds for $F\in \Sigma_1 (\rdd )$ and
$G\in W(L^{\Phi _1,\Phi _2}_{(\omega _2)})$.

\par

Let $a_{F*G}(k,\kappa)=\sup_{(x,\xi )\in (k,\kappa)+ Q_{2d}}
|(F*G)(x,\xi)\omega _0(x,\xi)|$, $F_{\omega _1}=F\cdot \omega _1$ and
$G_{\omega _2}=G \cdot \omega _2$, where $Q_d=[0,1]^d$ as usual.
First we estimate $a_{F*G}(k,\kappa)$ by
\begin{multline}\label{Eq:CoeffConvEst}
  a_{F*G}(k,\kappa) = \sup
  \left |
  \iint \limits _{\rdd} F(y,\eta)G(x-y,\xi-\eta)\omega _0(x,\xi)\, dy d\eta
  \right |
   \\[1ex]
   \lesssim
   \sup
   \left (
   \iint \limits _{\rdd}
   |F_{\omega _1}(y,\eta)|
   |G_{\omega _2}(x-y,\xi-\eta)|\, dy d\eta
   \right )
 \\[1ex]
  = \sup
  \left (
  \sum_{j,\iota\in \zd} \ \
  \iint \limits _{(j,\iota )+ Q_{2d}}
  |F_{\omega _1}(y,\eta)G_{\omega _2}(x-y,\xi-\eta)|\,dy d\eta
  \right )
  \\[1ex]
  \leq
  \sum_{j,\iota\in \zd} \ \
  \iint \limits _{(j,\iota )+ Q_{2d}}
  |F_{\omega _1}(y,\eta)|
  \left (
  \sup |G_{\omega _2}(x-y,\xi-\eta)|
  \right )dy d\eta
  \\[1ex]
  \leq
  \sum_{j,\iota \in \zd}
  \nm {F_{\omega _1}}{L^1((j,\iota )+Q_{2d})}
  \nm {G_{\omega _2}(\cdo +(k-j,\kappa-\iota))}{L^\infty ([-1,1]^{2d})}
  \\[1ex]
  = (b*c)(k,\kappa),
\end{multline}
where
\begin{align*}
  b(j,\iota) =   \nm {F_{\omega _1}}{L^1((j,\iota )+Q_{2d})}
\quad \text{and}\quad
  c(j,\iota) =
  \Vert G_{\omega _2} (\cdo +(j,\iota)) \Vert_{L^\infty ([-1,1]^{2d})}.
\end{align*}
Here the suprema in \eqref{Eq:CoeffConvEst} are taken with respect to
$(x,\xi )\in (k,\kappa) +Q_{2d}$.

\par

By \eqref{Eq:CoeffConvEst}, Corollary \ref{Cor:Young} and the fact that
$\Vert c \Vert_{\ell^{\Phi _1,\Phi _2}}
\simeq \Vert G \Vert_{W(L^{\Phi _1,\Phi _2})}$
we get
\begin{multline*}
  \Vert F*G \Vert_{W(L^{\Phi _1,\Phi _2}_{(\omega _0)})}
  = \Vert a_{F*G} \Vert|_{\ell^{\Phi _1,\Phi _2}}
  \leq
  \Vert b*c\Vert_{\ell^{\Phi _1,\Phi _2}}
  \leq
  \Vert b \Vert_{\ell^{r_0}}
  \Vert c \Vert_{\ell^{\Phi _1,\Phi _2}}
  \\[1ex]
  \lesssim \Vert F \Vert_{W(L^{1},L^{r_0}_{(\omega _1)})}
  \Vert G \Vert_{W(L^{\Phi _1,\Phi _2}_{(\omega _2)})},
\end{multline*}
and \eqref{2.9} follows.
\end{proof}

\par

By similar arguments we get the following semi-discrete convolution relation. The result
essentially generalizes \cite[Lemma 2.10]{GaSa}.

\par

\begin{lemma}\label{Lemma:No2.10}
Let $\Phi _j$, $\omega _j$ and $r_0$ be the same as in Lemma
\ref{Lemma:WienerConvEst}, and let
$\ep >0$.
Then the map $(a,F)\mapsto a*_{\ep } F$ from $\ell_0 (\ep \zdd)\times
(L^\infty (\rdd ) \cap \mascE'(\rdd ))$ to $L^\infty_{loc}(\rdd )$ extends uniquely
to a continuous mapping from $\ell^{\Phi _1,\Phi _2}_{(\omega _1)}(\ep \zdd)
\times W(L^{r_0}_{(\omega _2)})$ to $W(L^{\Phi _1,\Phi _2}_{(\omega _2)})$,
and
\begin{equation*}
  \nm {a*_{\ep } F } {W(L^{\Phi _1,\Phi _2}_{(\omega _0)})}
  \lesssim
  \Vert a \Vert_{\ell^{\Phi _1,\Phi _2}_{(\omega _1)}}
  \Vert F \Vert_{ W(L^{r_0}_{(\omega _2)})}.
\end{equation*}
\end{lemma}

\par

\begin{proof}
Let $v\in \mascP _E(\rr {2d})$ be submultiplicative and such that $\omega _1$ is
$v$-moderate.
If $a\in \ell^\infty_{(1/v)}(\ep \zdd)$ and $F\in L^\infty (\rdd )\cap \mascE'(\rdd )$, then
for $(x,\xi)\in \rdd$ belonging to a compact set, $(a*_{\ep } F)(x,\xi)$ is
given by a finite sum of locally bounded functions. This shows that $a*_{\ep } F$
is uniquely defined as an element in $L^\infty_{loc}(\rdd )$.

\par

In particular, since $\ell^{\Phi _1, \Phi _2}_{(\omega _1)}(\ep \zz {2d})\subseteq \ell^{\infty}
_{(1/v)}(\ep \zz {2d} )$, $ a*_{\ep } F$ is uniquely defined as an element in
$\ell^{\infty}_{(1/v)}(\ep \zz {2d})$ when $a\in
\ell^{\Phi _1, \Phi _2}_{(\omega _1)}(\ep \zz {2d} )$
and $F\in L^\infty(\rdd )\cap \mascE'(\rdd  )$.

\par

The result now follows by similar arguments as in the proof of Lemma
\ref{Lemma:WienerConvEst}, and the
fact that $L^\infty (\rdd )\cap \maclE'(\rdd )$ is dense in $W(L^{r_0}_{(v)})$.
The details are left for the reader.
\end{proof}

\par

\section{Gabor analysis of Orlicz modulation spaces}\label{sec4}

\par

In this section we extend the Gabor analysis in \cite{GaSa} to Orlicz
modulation spaces. We show that the quasi norm $f \mapsto \nm {V_{\phi _1} f}
{L^{\Phi _1, \Phi _2}_{(\omega)}}$ and $f \mapsto \nm {V_{\phi _2} f}
{W({L^{\Phi _1, \Phi _2}_{(\omega)}})}$ are equivalent when $\omega
\in \mascP_E (\rdd )$ and $\phi _1,\phi _2$ are suitable. (Cf. Proposition
\ref{Thm:OppWienEst1} below). This leads to that the analysis operator $C_{\phi _1}$
is continuous from $M^{\Phi _1, \Phi _2}_{(\omega)}(\rd )$ into
$\ell^{\Phi _1, \Phi _2}_{(\omega)}(\zdd )$, and that
the corresponding synthesis operator are continuous from
$\ell^{\Phi _1, \Phi _2}_{(\omega)}(\zdd )$ to $M^{\Phi _1, \Phi _2}_{(\omega)}(\rd )$.

\par

In the end we are able to prove that an element belongs to
$M^{\Phi _1, \Phi _2}_{(\omega)}(\rd )$,
if and only if its Gabor coefficients belong to
$\ell^{\Phi _1, \Phi _2}_{(\omega)}(\zdd)$. (Cf. Theorem \ref{Thm:Gabor}.)

\par

We also remark that our investigations are related to those general results in
\cite{Rau1,Rau2} by Rauhut on quasi-Banach coorbit space theory,
but remark that Rauhut's results do not cover our situation. For example,
our weight functions are allowed to grow and decay exponentially, which is
not the case in \cite{Rau1,Rau2}.

\par

\subsection{Comparisons between $\nm {V_{\phi _1} f}{L^{\Phi _1, \Phi _2}_{(\omega)}}$
and $\nm {V_{\phi _2} f} {W({L^{\Phi _1, \Phi _2}_{(\omega)}})}$}

\par

The next extension of \cite[Theorem 3.1]{GaSa} shows that the
condition $\nm{V_\phi f}{{L^{\Phi _1, \Phi _2}_{(\omega)}}}< \infty$
is independent of the choice of window function $\phi$, and that
different $\phi$ gives rise to equivalent norms.

\par

\begin{thm}\label{Thm:LebWindowTransf}
Let $\Phi _1,\Phi _2$ be quasi-Young functions of order $r_0 \in(0,1]$, $\omega, v
\in \mascP_E (\rdd )$ be such that $\omega$ is $v$-moderate, $\phi \in
\Sigma_1(\rd )\setminus{0}$ and let $\phi _{0}(x)=\pi^{-\frac d4}
e^{-\frac 12\cdot {|x|^{2}}}$ be the
standard Gaussian. Then
 \begin{align}
    \Vert V_{\phi _{0}}f \Vert_{L^{\Phi _1,\Phi _2}_{(\omega)}}
    &\lesssim
    \Vert V_{\phi _{0}}\psi \Vert_{L^{r_0}_{(v)}}
    \Vert V_{\phi}f \Vert_{L^{\Phi _1,\Phi _2}_{(\omega)}}
    \label{Eq:LebWindowTransf1}
\intertext{and}
    \Vert V_{\phi}f \Vert_{L^{\Phi _1,\Phi _2}_{(\omega)}}
    &\lesssim
    \Vert V_{\phi}\psi _{0} \Vert_{L^{r_0}_{(v)}}
    \Vert V_{\phi _{0}}f \Vert_{L^{\Phi _1,\Phi _2}_{(\omega)}},
    \label{Eq:LebWindowTransf2}
  \end{align}
  where $\psi$ and $\psi _{0}$ are canonical dual windows for $\phi$ and
  $\phi _{0}$ respectively with respect to some lattice $\ep  \zdd$.
\end{thm}

\par

\begin{proof}
Assume that $V_{\phi}f\in L^{\Phi _1,\Phi _2}_{(\omega)}(\rdd)$ and let
$\ep >0$ be such that
$$
\{e^{i\ep  \scal \cdo \kappa }
\phi(\cdo -\ep  k)\}_{k,\kappa\in\zd}
$$
is a Gabor frame for $L^{2}(\rd)$.
Let $v_0(x,\xi)=e^{|x|+|\xi|}$,
$\psi=(S_{\phi,\psi}^\ep )^{-1}\phi$ be the canonical dual window of $\phi$
and let $b(k, \kappa)=(V_\psi \phi _0)(\ep  k,\ep \kappa) $.
As a consequence of \cite[Theorem S]{Gc1} or the analysis in \cite[Chapter 13]{Gc2}
it follows that $\psi\in M^{r_0}_{(v)}(\rd)$ and
\begin{equation*}
\phi _0=\sum_{k,\kappa\in\zd }
b(k,\kappa)
\phi _{k,\kappa}, \quad \phi _{k,\kappa}(x)=e^{i\ep  \scal x \kappa} \phi(x-\ep  k),
\end{equation*}
with unconditional convergence in $M^1_{(v_0v)}(\rd) \subseteq M^{r_0}_{(v)}(\rd)$
(see also \cite[Proposition 1.4]{Toft16} for details). We have
\begin{multline*}
|V_{\phi _{0}}f(x,\xi)|
\le \sum_{k,\kappa\in\zd }|b(k,\kappa )V_{\phi _{k,\kappa}}f(x,\xi)|
\\[1ex]
\le
\sum_{k,\kappa\in\zd}
|b(k,\kappa)||V_\phi f(x+\ep  k,\xi+\ep \kappa)|
\\[1ex]
= (|\Check{b}|*_{\ep } |V_{\phi}f|)(x,\xi),
\end{multline*}
where $ \Check{b}(k,\kappa) = b(-k,-\kappa)$. By Corollary \ref{Cor:SemiDiscYoung}
and the fact that $\nm b{\ell ^{r_0}_{(v)}}\lesssim \nm {V_{\phi _0}\psi}{L^{q_0}_{(v)}}<\infty$,
in view of \cite[Theorem 3.7]{Toft15}, we obtain
\begin{multline*}
  \Vert V_{\phi _{0}}f\Vert_{L^{\Phi _1,\Phi _2}_{(\omega)}}
  \leq
  \Vert |\Check{b}|*_{\ep } |V_{\phi}f| \Vert _{L^{\Phi _1,\Phi _2}_{(\omega)}}
  \lesssim
  \Vert \Check{b}\Vert_{\ell^{r_0}_{(v)}} \Vert V_{\phi}f\Vert_{L^{\Phi _1,\Phi _2}_{(\omega)}}
  \\[1ex]
  \lesssim
 \Vert V_{\phi _0}\psi \Vert_{L^{r_0}_{(v)}}
 \Vert V_{\phi}f \Vert_{L^{\Phi _1,\Phi _2}_{(\omega)}}.
\end{multline*}
Here the last step follows from \cite[Lemma 3.2]{GaSa} (see also Proposition
\ref{Thm:OppWienEst1} below). This gives
\eqref{Eq:LebWindowTransf1}. By interchanging the roles of $\phi$
and $\phi _{0}$, we obtain \eqref{Eq:LebWindowTransf2}.
\end{proof}

\par

\begin{rem}
Let $\phi,\phi _0,\psi$ and $\psi _0$ be the same as in Theorem \ref{Thm:LebWindowTransf}.
By choosing the lattice dense enough, it follows that $V_{\phi _0} \psi \in
L_{(v)}^{r_0}(\rd)$ and $V_{\phi}\psi _0\in L_{(v)}^{r_0}(\rd)$. In fact, let
$v_0(x,\xi)$ be subexponential. Then $\phi _0, \phi\in M^{1}_{(v_0 v)}(\rd)
\subseteq M^{r_0}_{(v)}(\rd)$. By Theorem S in \cite{Gc1}, $\psi _0 , \psi
\in M^{1}_{(v_0 v)}(\rd)\subseteq M^{r_0}_{(v)}(\rd)$ provided that the
lattices of Gabor frames are dense enough. This implies that
$\nm {V_{\phi _0}\psi}{L^{r_0}_{(v)}}$ and $\nm {V_{\phi}\psi _0}{L^{r_0}_{(v)}}$
in \eqref{Eq:LebWindowTransf1} and \eqref{Eq:LebWindowTransf2} are finite.
\end{rem}

\par

\begin{prop}\label{Thm:OppWienEst1}
Let $\Phi _1,\Phi _2$ be quasi-Young functions of order $r_0 \in(0,1]$, $\omega
\in \mascP_E (\rdd )$, $f\in M^{\Phi _1,\Phi _2}_{(\omega)}(\rd)$ and let $\phi _{0}(x)
=\pi^{-\frac d4}e^{-\frac{|x|^{2}}{2}}$.
Then $V_{\phi _0}f\in W(L^{\Phi _1,\Phi _2}_{(\omega)})$ and
  \begin{equation}\label{WSTF}
    \nm {V_{\phi _0} f}{{W(L^{\Phi _1,\Phi _2}_{(\omega)})}}
    \lesssim
    \Vert V_{\phi _0}f\Vert_{L^{\Phi _1,\Phi _2}_{(\omega)}}.
  \end{equation}
\end{prop}

\par

\begin{proof}
Let $F_0 = |V_{\phi _0}f|$ and let
\begin{equation*}
a_{F_0}(k,\kappa) = \sup_{(x,\xi)\in Q_{2d}}
F_0(x+k,\xi+\kappa),\qquad k,\kappa \in \zz d.
\end{equation*}
For each $k,\kappa\in \zd $, choose
$$
X_{k,\kappa}=(x_{k,\kappa},\xi_{k,\kappa})\in (k,\kappa)+Q_{2d}
$$
such that
\begin{equation*}
F_0(X_{k,\kappa})
= a_{F_0}(k,\kappa).
\end{equation*}
We have
\begin{equation} \label{equ1}
  \Vert f\Vert_{M^{\Phi _1,\Phi _2}_{(\omega)}} =
\Vert F_0^{r_0}\Vert_{L^{\Phi _{0,1},\Phi _{0,2}}_{(\omega)}}^{1/r_0}.
\end{equation}
For any
\begin{align*}
X &= (x_{1},\dots,x_{d},\xi_{1},\dots,\xi_{d})\in \rdd
\intertext{and}
X_0 &= (x_{0,1},\dots,x_{0,d},\xi_{0,1},\dots,\xi_{0,d})\in \rdd ,
\end{align*}
let $X_j= (x_j,\xi _j)\in \rr 2$, $X_{0,j}= (x_{0,j},\xi _{0,j})\in \rr 2$, $j=1,\dots ,d$,
$D_r(X_0)$ be polydisc
$$
\sets {X\in \rr d}
{|X_j-X_{0,j}|<r, j=1,\dots ,d},
$$
$U_{1,d} = [-r, 1+r]^d$ and $U_{2,d}=[-2-r,2+r]^d$.
By Lemma 2.3 in \cite{GaSa} we get
\begin{multline*}
F_0 (X_{k,\kappa})^{r_0}\omega (X_{k,\kappa})^{r_0}
\lesssim  \iint \limits_{D_r(X_{k,\kappa})}F_0(x,\xi)^{r_0}
 \omega (X_{k,\kappa})^{r_0}\, dx d\xi
\\[1ex]
\lesssim  \iint \limits_{X_{k,\kappa}+U_{1,2d}} F_0(x,\xi)^{r_0}
\omega(x,\xi)^{r_0}\, dx d\xi.
\end{multline*}

\par

In order to estimate the left hand side of \eqref{WSTF} we apply
the latter estimates on
\begin{multline} \label{Eq:est1}
  \sum_{k\in \zd} \Phi _{0,1} \left(
  \frac{F_0(X_{k,\kappa})^{r_0} \omega(X_{k,\kappa})^{r_0}}{\lambda^{r_0}}
  \right)
  \\[1ex]
  \le
  \sum_{k\in \zd}  \Phi _{0,1} \left(
  \frac{C^{r_0}}{\lambda^{r_0}} \iint\limits_{X_{k,\kappa}+ U_{1,2d}}
  F_0(x,\xi)^{r_0} \omega(x,\xi)^{r_0}\, dx d\xi
  \right)
  \\[1ex]
 \leq
 \sum_{k\in \zd} \Phi _{0,1} \left(
  \frac{C^{r_0}}{\lambda^{r_0}}
  \iint \limits_{(k,\kappa)+U_{2,2d}} F_0(x,\xi)^{r_0} \omega(x,\xi)^{r_0}\, dx d\xi
   \right),
\end{multline}
which is true for some $C>0$. Since the volume of $U_{2,d}$ is equal to
$(4+2r)^d$ and $\Phi _{0,1}$ is convex, Jensen's inequality gives
  \begin{multline} \label{Eq:est2}
  \sum_{k\in \zd} \Phi _{0,1} \left(
  \frac{C^{r_0}}{\lambda^{r_0}} \iint\limits_{(k,\kappa)+U_{2,2d}} F_0(x,\xi)^{r_0}
  \omega(x,\xi)^{r_0}\, dx d\xi
   \right)
   \\[1ex]
   \le
   \sum_{k\in \zd}\ \
   (4+2r)^{-d}
    \int\limits_{k+U_{2,d}}\Phi _{0,1}
    \left(
  \frac{  C^{r_0}(4+2r)^d}{\lambda^{r_0}}
  \int \limits _{\kappa +U_{2,d}}
  F_0(x,\xi)^{r_0}
 \omega(x,\xi)^{r_0}\, d\xi
  \right)dx
  \\[1ex]
  = 4^d (4+2r)^{-d}
  \int \limits _{\rd} \Phi _{0,1} \left(
  \frac{  C^{r_0}(4+2r)^d}{\lambda^{r_0}}
  \int \limits _{\kappa +U_{2,d}} F_0(x,\xi )^{r_0}  \omega(x,\xi )^{r_0}\, d\xi
  \right)dx.
\end{multline}
By \eqref{Eq:est1}, \eqref{Eq:est2} and the definition of $L^{\Phi _{0,1}}$
norm we get
\begin{equation*}
\Vert a(\cdo, \kappa)\Vert_{\ell^{\Phi _1}_{(\omega)}}
\lesssim
\left\Vert \,
\int\limits_{\kappa+U_{2,d}}
F_0(\cdo,\xi)^{r_0} \omega(\cdo,\xi)^{r_0}\, d\xi
\right \Vert^{1/r_0}_{L^{\Phi _{0,1}}}.
\end{equation*}

\par

Let $\Vert a(\cdo , \kappa) \Vert_{\ell^{\Phi _1}_{(\omega)}} = b(\kappa)$.
Then by Minkowski's inequality and again using Jensen's inequality we get
for some $C>0$ that
\begin{multline*}
    \sum_{\kappa\in \zd} \Phi _{0,2}\left(\frac{b(\kappa)^{r_0}}{\lambda^{r_0}}
    \right)
    \le
    \sum_{\kappa\in \zd} \Phi _{0,2}\left (
    \left \Vert
    \frac{C^{r_0}}{\lambda^{r_0}} \int\limits_{\kappa+U_{2,d}}
    F_0(\cdo,\xi)^{r_0}
    \omega(\cdo,\xi)^{r_0}\, d\xi
    \right \Vert_{L^{\Phi _{0,1}}}
      \right)
      \\[1ex]
      \leq
      \sum _{\kappa\in \zd} \Phi _{0,2}\left( \frac{C^{r_0}}{\lambda^{r_0}}
      \int \limits _{\kappa+U_{2,d}}
      \Vert F_0(\cdo,\xi)^{r_0}
      \omega(\cdo,\xi)^{r_0}
      \Vert_{L^{\Phi _{0,1}}}\, d\xi
      \right)
      \\[1ex]
      \le
      \sum_{\kappa\in \zd}
      (4+2r)^{-d}
      \int\limits_{\kappa+U_{2,d}} \Phi _{0,2}
      \left(
      \frac{C^{r_0}(4+2r)^d}{\lambda^{r_0}}
      \Vert F_0(\cdo,\xi)^{r_0}
      \omega(\cdo,\xi)^{r_0}
      \Vert_{L^{\Phi _{0,1}}}
      \right)\, d\xi
      \\[1ex]
      =
    4^d (4+2r)^{-d}
     \int\limits_{\rd} \Phi _{0,2}
     \left(
     \frac{C^{r_0}(4+2r)^d}{\lambda^{r_0}}
    \Vert F_0(\cdo,\xi)^{r_0}
    \omega(\cdo,\xi)^{r_0}
    \Vert_{L^{\Phi _{0,1}}}
    \right) d\xi.
\end{multline*}
By the definition of $L^{\Phi _{0,2}}_{(\omega)}$ norm we get
\begin{equation} \label{equ2}
\Vert a\Vert_{\ell^{\Phi _1,\Phi _2}_{(\omega)}}
= \Vert b\Vert_{\ell^{\Phi _2}}
\lesssim
\Vert F_0\Vert_{L^{\Phi _1,\Phi _2}_{(\omega)}}.
\end{equation}
Hence \eqref{equ1} and \eqref{equ2} give
\begin{equation*}
   \Vert  V_{\phi _0}f\Vert_{W(L^{\Phi _1,\Phi _2}_{(\omega)})}
   = \Vert a\Vert_{\ell^{\Phi _1,\Phi _2}_{(\omega)}}
   \lesssim\Vert F_0\Vert_{L^{\Phi _1,\Phi _2}_{(\omega)}}
   =\Vert V_{\phi _0}f\Vert_{L^{\Phi _1,\Phi _2}_{(\omega)}}.\qedhere
\end{equation*}
\end{proof}

\par

We have now the following extension of Proposition
\ref{Thm:OppWienEst1}, \cite[Theorem 3.3]{GaSa} and
\cite[Proposition 3.4]{Toft15}.

\par

\begin{thm}\label{Thm:ModWienerInv}
Let $\Phi _1,\Phi _2$ be quasi-Young functions of order $r_0 \in(0,1]$,
$\omega, v \in \mascP_E (\rdd )$ be such that $\omega$ is $v$-moderate,
and let $\phi _1,\phi _2\in M^{r_0}_{(v)}(\rd)$ with dual windows in  $M^{r_0}_{(v)}(\rd)$
with respect to some lattice in $\rdd $.
If $V_{\phi _1} f\in L^{\Phi _1,\Phi _2}_{(\omega)}(\rdd )$,
then $V_{\phi _2} f\in W(L^{\Phi _1,\Phi _2}_{(\omega)})$ and
\begin{equation*}
\nm {V_{\phi _2} f}{W(L^{\Phi _1,\Phi _2}_{(\omega)})}
\lesssim
\nm {V_{\phi _1} f}{L^{\Phi _1,\Phi _2}_{(\omega)}}.
\end{equation*}
\end{thm}

\par

\begin{proof}
Using the reproducing formula we have
\begin{equation*}
 |V_{\phi _2}f(x,\xi)|\leq
 \frac{1}{\Vert \phi _0\Vert^{2}_2}\left(
 |V_{\phi _0}f|*|V_{\phi _2} \phi _0|
 \right) (x,\xi).
\end{equation*}
By Lemma \ref{Lemma:WienerConvEst}, Theorem \ref{Thm:LebWindowTransf}
and Proposition \ref{Thm:OppWienEst1} we obtain
\begin{multline*}
\Vert V_{\phi _2}  f \Vert_{W(L^{\Phi _1,\Phi _2}_{(\omega)})}
\lesssim
\Vert V_{\phi _0} f
\Vert_{W(L^{\Phi _1,\Phi _2}_{(\omega)})}
\Vert V_{\phi _2}  \phi _0\Vert_{W(L^1,L^{r_0}_{(v)})}
\\[1ex]
\lesssim
\Vert V_{\phi _0} f \Vert_{L^{\Phi _1,\Phi _2}_{(\omega)}}
\Vert V_{\phi _2}  \phi _0\Vert_{W(L^1,L^{r_0}_{(v)})}.
\end{multline*}
By \cite[Proposition 1.15$'$]{Toft22} we get $\nm {V_{\phi _2}  \phi _0}{W(L^1,L^{r_0}_{(v)})}
\asymp \nm {\phi _2}{M^{r_0}_{(v)}}$. (See also \cite{Toft15}.) Hence, if $\psi _1$
is the dual window of
$\phi _1$, then Theorem \ref{Thm:LebWindowTransf} gives
\begin{equation*}
\Vert V_{\phi _2}  f \Vert_{W(L^{\Phi _1,\Phi _2}_{(\omega)})}
\lesssim
\nm {\phi _2}{M^{r_0}_{(v)}}
\Vert V_{\phi _0} \psi\Vert_{L^{r_0}_{(v)}}
\Vert V_{\phi _1}  f \Vert_{L^{\Phi _1,\Phi _2}_{(\omega)}}
\\[1ex]
\asymp
\Vert V_{\phi _1} f \Vert_{L^{\Phi _1,\Phi _2}_{(\omega)}}.\qedhere
\end{equation*}
\end{proof}

\par

We may now deduce suitable continuity properties for analysis and
synthesis operators. (Cf. \cite[Theorem 3.5]{GaSa} and \cite[Theorem 5.6]{Rau2})

\par

\begin{thm}\label{Thm:AnalysisOp}
Let $\ep >0$, $\Phi _1,\Phi _2$ be quasi-Young functions of order $r_0 \in(0,1]$, $\omega, v \in
\mascP_E (\rdd )$ be such that $\omega$ is $v$-moderate and let $\phi
\in M^{r_0}_{(v)}(\rd)$ with dual windows in $M^{r_0}_{(v)}(\rd)$ with respect to
$\ep \zz {2d}$.
Then the analysis operator $C_{\phi}^{\ep}$ is continuous from
$M^{\Phi _1,\Phi _2}_{(v)}(\rd)$ into $\ell^{\Phi _1,\Phi _2}_{(\omega)}(\zdd)$, and
$$
\nm {C_{\phi}^{\ep}f}{\ell^{\Phi _1,\Phi _2}_{(\omega)}}
\lesssim
\nm f{ M_{(\omega)}^{\Phi _1,\Phi _2}},
\quad
f\in M_{(\omega)}^{\Phi _1,\Phi _2}(\rd).
$$
\end{thm}

\par

\begin{proof}
Since $V_{\phi}f$ is continuous, we have by Lemma \ref{Lemma:No2.8} with $\alpha=\beta=
\ep $ and Theorem \ref{Thm:ModWienerInv} that
\begin{multline*}
  \Vert C_{\phi}^{\ep }
  f\Vert_{\ell^{\Phi _1,\Phi _2}_{(\omega)}}
  = \Vert V_\phi f(\ep \cdo)\Vert_{\ell^{\Phi _1,\Phi _2}_{(\omega)}}
  \lesssim
  \Vert V_\phi f\Vert_{W(L^{\Phi _1,\Phi _2}_{(\omega)})}
  \\[1ex]
  \lesssim
 \Vert V_\phi f\Vert_{L^{\Phi _1,\Phi _2}_{(\omega)}}
  \asymp \Vert f\Vert_{M^{\Phi _1,\Phi _2}_{(\omega)}},
\end{multline*}
which completes the proof.
\end{proof}

\par

\begin{thm}\label{Thm:SynthesisOp}
Let $\ep >0$, $\Phi _1,\Phi _2$ be quasi-Young functions of order $r_0 \in(0,1]$,
$\omega, v \in \mascP_E (\rdd )$ be such that $\omega$ is $v$-moderate
and let $\psi \in M^{r_0}_{(v)}(\rd)$ with dual window in $M^{r_0}_{(v)}(\rd)$ with respect to
$\ep \zz {2d}$.
Then the synthesis operator $D_{\psi}^{\ep }$ is continuous from
$\ell^{\Phi _1,\Phi _2}_{(\omega)}(\zdd)$ into
$M^{\Phi _1,\Phi _2}_{(\omega)}(\rd)$, and
$$
\Vert D_{\psi}^{\ep }
c\Vert_{M^{\Phi _1,\Phi _2}_{(\omega)}}
\lesssim
\Vert c\Vert_{\ell^{\Phi _1,\Phi _2}_{(\omega)}},
\quad
c\in \ell^{\Phi _1,\Phi _2}_{(\omega)}(\zdd).
$$
\end{thm}

\par

\begin{proof}
Let $\phi _{0}$ be the standard Gaussian window. We have to
show that $V_{\phi _{0}}(D_{\psi}^{\ep }c)
\in L^{\Phi _1,\Phi _2}_{(\omega)}(\rdd )$ when
$c\in \ell^{\Phi _1,\Phi _2}_{(\omega)}(\zdd)$.
Since
$$
\left(
V_{\phi _{0}}
(e^{i\ep \scal \cdo \kappa} \psi(\cdo-\ep  k))
\right)
(x,\xi)
=V_{\phi _{0}}\psi(x-\ep  k,\xi-\ep \kappa),
$$
we get
\begin{multline*}
|V_{\phi _{0}}(D_\psi^{\ep }c)(x,\xi)|
=\left|
|V_{\phi _{0}}
\left(
\sum_{k,\kappa\in\zd}
c(k,\kappa)
(e^{i\ep \scal \cdo \kappa} \psi(\cdo-\ep  k)
\right)(x,\xi)
\right|
\\[1ex]
= \left|
\sum_{k,\kappa\in\zd}
c(k,\kappa)\left(V_{\phi _{0}}
(e^{i\ep \scal \cdo \kappa} \psi(\cdo-\ep  k))
\right)(x,\xi)
\right|
\\[1ex]
= \left|\sum_{k,\kappa\in\zd}
c(k,\kappa)(V_{\phi _{0}}\psi)(x-\ep  k,\xi-\ep  \kappa)\right|
\\[1ex]
\leq \sum_{k,\kappa\in\zd}
\left|c(k,\kappa)(V_{\phi _{0}}\psi)(x-\ep  k,\xi-\ep  \kappa)
\right|
=(|c|*_{\ep } |V_{\phi _{0}}\psi|)(x,\xi).
\end{multline*}
and Lemma \ref{Lemma:No2.10} implies that
\begin{multline*}
\Vert D_{\psi}^{\ep }
c\Vert_{M^{\Phi _1,\Phi _2}_{(\omega)}}=
\Vert V_{\phi _0}(D_{\psi}^{\ep } c)\Vert_{L^{\Phi _1,\Phi _2}_{(\omega)}}
\leq
\Vert V_{\phi _0}(D_{\psi}^{\ep } c)\Vert _{W(L^{\Phi _1,\Phi _2}_{(\omega)})}
\\[1ex]
\leq \Vert|c|*_{\ep } |V_{\phi _{0}}\psi|\Vert_{W(L^{\Phi _1,\Phi _2}_{(\omega)})}
\leq
C\Vert c\Vert_{\ell^{\Phi _1,\Phi _2}_{(\omega)}}
\Vert V_{\phi _0}\psi\Vert_{W(L^{r_0}_{(v)})},
\end{multline*}
and the result follows from Theorem \ref{Thm:ModWienerInv}.
\end{proof}

\par

The next theorem is the main result of the section, and shows that the Gabor
analysis in \cite{GaSa} for modulation spaces also holds for quasi-Orlicz
modulation spaces.

\par

\begin{thm}\label{Thm:Gabor}
Let $\ep >0$, $\Phi _1,\Phi _2$ be quasi-Young functions of order
$r_0 \in(0,1]$, $\omega, v \in \mascP_E (\rdd )$ be such that $\omega$
is $v$-moderate and let $\phi,\psi\in M^{r_0}_{(v)}(\rd)$ be such that
\begin{equation}\label{Eq:DualFrames}
\{ e^{i\ep \scal \cdo \kappa}\phi (\cdo -\ep k) \} _{k,\kappa \in \zz d}
\quad \text{and}\quad
\{ e^{i\ep \scal \cdo \kappa}\psi (\cdo -\ep k) \} _{k,\kappa \in \zz d}
\end{equation}
are dual frames to each others. Then the following is true:
\begin{enumerate}
\item The Gabor frame operator $S_{\phi,\psi}^\ep  = D_{\psi}^{\ep } \circ
    C_{\phi}^{\ep }$
    is the identity operator on $M_{(\omega)}^{\Phi _1, \Phi _2}(\rd)$;

\vrum

\item If $f\in M_{(\omega)}^{\Phi _1, \Phi _2}(\rd)$, then
\begin{align*}
f&=\sum_{k,\kappa\in\zd }
(V_\psi f)(\ep  k,\ep \kappa)
e^{i\ep  \scal \cdo \kappa }\phi(\cdo -\ep  k)
\\[1ex]
&=\sum_{k,\kappa\in\zd }
(V_\phi f)(\ep  k,\ep \kappa)
e^{i\ep  \scal \cdo \kappa }\psi(\cdo -\ep  k),
\end{align*}
with unconditionally convergence in $M^{\Phi _1,\Phi _2}_{(\omega)}(\rd)$
when
$\mascS(\rdd )$ is dense in $L^{\Phi _1,\Phi _2}(\rdd )$,
and with convergence in $M^\infty_{(\omega)}(\rd)$ with respect to
the weak$^*$ topology otherwise.
\end{enumerate}

\par

Furthermore,
\begin{multline}\label{Eq:EquivModNormDisc}
\Vert \{ (V_\phi f)(\ep  k,\ep  \kappa)\}_{k,\kappa\in \zd}
\Vert_{\ell_{(\omega)}^{\Phi _1, \Phi _2}}
\asymp
\Vert \{ (V_\psi f)(\ep  k,\ep  \kappa)\}_{k,\kappa\in \zd}
\Vert_{\ell_{(\omega)}^{\Phi _1, \Phi _2}}
\\[1ex]
\asymp \Vert f \Vert_{M^{\Phi _1, \Phi _2}_{(\omega)}}.
\end{multline}
\end{thm}

\par

\begin{proof}
Since \eqref{Eq:DualFrames} are dual frames, it follows that
$D_{\psi}^{\ep }\circ C_{\phi}^{\ep }$ is the identity operator
on $M^\infty _{(\omega )}(\rr d)$, in view of \cite[Corollary 12.2.6]{Gc2}. A combination
of this fact and $M^{\Phi _1,\Phi _2}_{(\omega)}(\rd)\subseteq M^\infty _{(\omega )}(\rr d)$
shows that $f=D_{\psi}^{\ep }\circ
C_{\phi}^{\ep }$ holds for
all $f\in M^{\Phi _1,\Phi _2}_{(\omega)}(\rd)$. By Theorems
\ref{Thm:AnalysisOp} and \ref{Thm:SynthesisOp}, the norm equivalence
between the first and last
expressions in \eqref{Eq:EquivModNormDisc} follows from
\begin{multline*}
\Vert f\Vert_{M^{\Phi _1,\Phi _2}_{(\omega)}}=
\Vert( D_{\psi}^{\ep }\circ C_{\phi}^{\ep }) f\Vert_{M^{\Phi _1,\Phi _2}_{(\omega)}}
\leq
\Vert D_{\psi}^{\ep }
\Vert_{\maclB(\ell^{\Phi _1,\Phi _2}_{(\omega)},M^{\Phi _1,\Phi _2}_{(\omega)})}
\Vert C_{\phi}^{\ep } f\Vert_{\ell^{\Phi _1,\Phi _2}_{(\omega)}}
\\[1ex]
\leq \Vert D_{\psi}^{\ep }
\Vert_{\maclB(\ell^{\Phi _1,\Phi _2}_{(\omega)},
M^{\Phi _1,\Phi _2}_{(\omega)})} \Vert C_{\phi}^{\ep }
\Vert_{\maclB(M^{\Phi _1,\Phi _2}_{(\omega)},
\ell^{\Phi _1,\Phi _2}_{(\omega)})} \Vert f\Vert_{M^{\Phi _1,\Phi _2}_{(\omega)}}.
\end{multline*}
By interchanging the roles for $\phi$ and $\psi$, we deduce the other relations
in \eqref{Eq:EquivModNormDisc}.
\end{proof}

\par

\begin{rem}
Let $\omega\in \mascP_E (\rdd )$, $\Phi _{0,1},\Phi _{0,2}$ be Young functions,
and $\Phi _1$ and $\Phi _2$ be quasi-Young functions of order $r_0\in (0,1]$
with respect to $\Phi _{0,1}$ and $\Phi _{0,2}$, respectively.

\par

Since $\mascS(\rdd )$ is continuously embedded in $L^{\Phi _1, \Phi _2}(\rdd )$,
and that $\Sigma_1(\rdd )$ is dense in $\mascS(\rdd )$, it follows that $\Sigma
_1(\rdd )$ is dense in $L^{\Phi _1, \Phi _2}(\rdd )$ when $\mascS(\rdd )$ is dense
in $L^{\Phi _1, \Phi _2}(\rdd )$, according to (2) in Theorem \ref{Thm:Gabor}.

\par

By straight-forward computations it follows that $\Sigma_1(\rdd )$ is dense in
$L^{\Phi _1, \Phi _2}_{(\omega)}(\rdd )$ when $\mascS(\rdd )$ is dense in
$L^{\Phi _1, \Phi _2}(\rdd )$.

\par

A sufficient condition for $\mascS(\rdd )$ to be dense in $L^{\Phi _{0,1},\Phi _{0,2}}(\rdd )$
and in $L^{\Phi _{1},\Phi _{2}}(\rdd )$ is that $\Phi _{0,1}$ and $\Phi _{0,2}$ fullfill the
so-called $\Delta_2$-condition in \cite{SchF}. In particular, this is true when
$\Phi _j(t)\gtrsim t^{\theta},\ j=1,2$ near the origin, for some $\theta >0$.
\end{rem}

\par

\subsection{Some consequences}
Next we present some consequences of the previous results, and begin with
the following invariance of the $M^{\Phi _1,\Phi _2}_{(\omega)}(\rd)$ norm
with respect to the choice of $\phi _1$ and $\phi _2$ in
Theorem \ref{Thm:ModWienerInv}.

\par

\begin{thm}
Let $\Phi _1,\Phi _2$ be quasi-Young functions of order $r_0 \in(0,1]$,
$\omega, v \in \mascP_E (\rdd )$ be such that $\omega$ is $v$-moderate,
and let $\phi \in M^{r_0}_{(v)}(\rd)$ with dual window in $M^{r_0}_{(v)}(\rd)$.
Then $f\mapsto \nm {V_\phi f}{L^{\Phi _1,\Phi _2}_{(\omega)}}$ and
$f \mapsto \nm {V_\phi f}{W(L^{\Phi _1,\Phi _2}_{(\omega)})}$ are quasi-norms on $\Sigma_1 '(\rd)$
which are equivalent to the quasi-norm $f\mapsto \nm {f}{M^{\Phi _1,\Phi _2}_{(\omega)}}$.
\end{thm}
Recall that $\nm {f}{M^{\Phi _1,\Phi _2}_{(\omega)}}
= \nm {V_{\phi _0} f}{L^{\Phi _1,\Phi _2}_{(\omega)}}$,
when $\phi _0(x)=\pi^{-\frac d4}e^{-\frac 12 |x|^2}$.

\par

\begin{proof}
The result is an immediate consequence of \eqref{Eq:Wiener} and
Theorem \ref{Thm:ModWienerInv}.
\end{proof}

\par

\begin{thm}\label{Thm:OrliczModInvariance}
Suppose that $\omega\in \mascP_E(\rdd )$, $\Phi _k$ and $\Psi _k$ are quasi
Young functions such that
\begin{equation} \label{Eq:delta}
    \lim \limits _{t\to 0^+}\frac{\Psi _k(t)}{\Phi _k(t)}
\end{equation}
exist and are finite, $k=1,2$. Then the following is true:
\begin{enumerate}
    \item $\ell^{\Phi _1, \Phi _2}_{(\omega)}(\zdd)$
    is continuously embedded in $\ell^{\Psi _1, \Psi _2}_{(\omega)}(\zdd)$;

\vrum

    \item $M^{\Phi _1,\Phi _2}_{(\omega)}(\rd)$ is continuously embedded
    in $M^{\Psi _1,\Psi _2}_{(\omega)}(\rd)$.
\end{enumerate}
\end{thm}

\par

\begin{proof}
By Theorem \ref{Thm:Gabor} it suffices to prove (1). Since $a\mapsto a\cdot
\omega$ is isometric bijection from $\ell^{\Phi _1, \Phi _2}_{(\omega)}(\zdd)$
to $\ell^{\Phi _1, \Phi _2}(\zdd)$, we may assume that $\omega=1$. In view
of \eqref{Eq:delta}, there is a $t_0>0$ such that
$$
\Psi _k(t)\lesssim\Phi _k(t), \quad 0\le t\le t_0,\, k=1,2.
$$
Let
\begin{align*}
    \Phi _{*,k}(t)=
    \begin{cases}
      \Phi _k(t), & 0\le t\le t_0,
      \\[1ex]
      \infty, &t>t_0
    \end{cases}
\end{align*}
and
\begin{align*}
    \Psi _{*,k}(t)=
    \begin{cases}
      \Psi _k(t), & 0\le t\le t_0,
      \\[1ex]
      \infty, &t>t_0.
    \end{cases}
\end{align*}
We claim
\begin{equation}\label{Eq:last}
\ell^{\Phi _1, \Phi _2}(\zdd)=\ell^{{\Phi}_{*,1}, {\Phi}_{*,2}}(\zdd),
\end{equation}
also in topological sense.

\par

In fact, let
$$
\Phi _{*,k}^0(t) = \Phi _{*,k} (t^{\frac 1{r_0}})
\quad \text{and}\quad
\Phi _{k}^0(t) = \Phi _{k} (t^{\frac 1{r_0}}),\quad k=1,2.
$$
Then $\Phi _{*,k}^0$ and $\Phi _{k}^0$ are Young functions, and by
Proposition \ref{Prop:ModPhiPsi} we have
$$
\nm a{\ell^{{\Phi}_{*,1}^0, {\Phi}_{*,2}^0}} \asymp \nm a{\ell^{{\Phi}_{1}^0, {\Phi}_{2}^0}}.
$$
This gives
\begin{equation}\label{Eq:Last2}
  \nm a{\ell^{{\Phi}_{*,1}, {\Phi}_{*,2}}} =
\nm {|a|^{r_0}}{\ell^{{\Phi}_{*,1}^0, {\Phi}_{*,2}^0}}^{\frac 1{r_0}}
\asymp
\nm {|a|^{r_0}}{\ell ^{{\Phi}_{1}^0, {\Phi}_{2}^0}}^{\frac 1{r_0}}
\asymp
\nm a{\ell^{{\Phi}_{1}, {\Phi}_{2}}},
\end{equation}
and \eqref{Eq:last} follows.

\par

By \eqref{Eq:Last2} and the fact that
$$
\Psi _{*,k}(t)\lesssim \Phi _{*,k}(t),\quad t\in \rd_+,
$$
we get
$$
\Vert a\Vert_{\ell^{\Psi _1, \Psi _2}}
\le \Vert a\Vert_{\ell^{\Psi _{*,1}, \Psi _{*,2}}}
\lesssim \Vert a\Vert_{\ell^{\Phi _{*,1} \Phi _{*,2}}}
\asymp \Vert a\Vert_{\ell^{\Phi _1, \Phi _2}},
$$
and the result follows.
\end{proof}

\par

By Theorem \ref{Thm:OrliczModInvariance} and its proof we may now
extend Proposition \ref{Prop:ModPhiPsi} to the quasi-Banach case as
follows. The details are left for the reader.

\par

\begin{prop}\label{Prop:OrliczModInvariance}
Let $\Phi _j, \Psi _j,\ j=1,2$ be quasi-Young functions and $\omega\in\mascP_E(\rdd )$.
Then the following conditions are equivalent:
\begin{enumerate}
    \item $M^{\Phi _{1},\Phi _{2}}_{(\omega)}(\rd)\subseteq
    M^{\Psi _{1}, \Psi _{2}}_{(\omega)}(\rd)$;

    \vrum

    \item $\ell^{\Phi _{1},\Phi _{2}}_{(\omega)}(\zdd)\subseteq
    \ell^{\Psi _{1}, \Psi _{2}}_{(\omega)}(\zdd)$;

    \vrum

    \item there is a constant $t_0 >0$ such that $\Psi _{j} (t)
    \lesssim  \Phi _{j} (t)$ for all $0\leq t\leq t_0$.
\end{enumerate}
\end{prop}

\par

Next we discuss compactness of Orlicz modulation spaces.
The following result follows by similar arguments as for \cite[Theorem 3.9]{PfTo},
using the fact that $M^{\Phi _1,\Phi _2}_{(\omega)} (\rd)$ is continuously embedded
in $M^\infty_{(\omega)} (\rd)$ in view of Corollary \ref{Cor:ModWienerInv}.
The details are left for the reader.
\begin{thm}
Let $\omega_1,\omega_2 \in \mascP_E(\rdd )$. Then the injection map
    \begin{equation*}
           i: M^{\Phi _1,\Phi _2}_{(\omega_1)} (\rd)
           \to
           M^{\Phi _1,\Phi _2}_{(\omega_2)}(\rd)
    \end{equation*}
is compact if and only if
$$
\lim _{|X|\to \infty} \frac {\omega_2(X)}{\omega_1(X)}=0.
$$
\end{thm}

\par

\end{document}